\documentclass[letterpaper,reqno]{amsart}
\usepackage{mathtools}
\usepackage{xcolor,amsfonts,amssymb}
\usepackage{comment}
\usepackage{cancel}
\usepackage{tikz}
\usepackage{pgf}
\usepackage[colorlinks=true,pdfpagemode=UseNone,urlcolor=blue,linkcolor=blue,citecolor=blue,backref=page]{hyperref}

\mathtoolsset{showonlyrefs}

\usetikzlibrary{calc}

\newtheorem{theorem}{Theorem}
\newtheorem{lemma}[theorem]{Lemma}
\newtheorem{corollary}[theorem]{Corollary}
\newtheorem{proposition}[theorem]{Proposition}

\newtheorem{definition}{Definition}

\title{Singularly Weighted X-ray Tensor Tomography}

\author{Jonathan Kay}
\thanks{Department of Mathematics, University of California, Santa Cruz CA 95064; email: jonkay@ucsc.edu}
\author{Fran\c{c}ois Monard}
\thanks{Department of Mathematics, University of California, Santa Cruz CA 95064; email: fmonard@ucsc.edu}

\newcommand\inward{{\partial_+S\mathbb{D}}}

\newcommand\SD{{S\mathbb{D}}}

\newcommand\Dm{\mathbb{D}}
\newcommand\Rm{\mathbb{R}}
\newcommand\Zm{\mathbb{Z}}
\newcommand\dbar{\overline{\partial}}
\renewcommand\d{\mathrm{d}}
\newcommand\tr{\operatorname{tr}}
\newcommand\Id{\operatorname{Id}}
\newcommand\Stt{S_\mathrm{tt}}

\newcommand\pif{\pi_F}

\newcommand\fitt{f^{\mathrm{itt}}}
\newcommand\zbar{{\bar z}}
\newcommand\wtZ{{\widehat Z}}

\newcommand \noteJ[1] {{\color{olive} \ifmmode \text{[J: #1]} \else [J: #1] \fi}}

\calclayout

\date{\today}

\begin{document}

\begin{abstract}
If $d$ is a boundary defining function for the Euclidean unit disk and $I$ denotes the geodesic X-ray transform, for $\gamma\in (-1,1)$, we study the singularly-weighted X-ray transforms $I_m d^\gamma$ acting on symmetric $m$-tensors. For any $m$, we provide a sharp range decomposition and characterization in terms of a distinguished Hilbert basis of the data space, that comes from earlier studies of the Singular Value Decomposition for the case $m=0$, see \cite{Mishra2022}. Since for $m\ge 1$, the transform considered has an infinite-dimensional kernel, we fully characterize this kernel, and propose a representative for an $m$-tensor to be reconstructed modulo kernel, along with efficient procedures to do so. This representative is based on a new generalization of the potential/conformal/transverse-tracefree decomposition of tensor fields in the context of singularly weighted $L^2$-topologies. 
\end{abstract}
\maketitle


\section{Introduction}
We study a family of singularly weighted X-ray transforms that act on tensor fields over the unit disk $\Dm = \{(x,y)\in \Rm^2,\ x^2+y^2\le 1\}$ equipped with the Euclidean metric $g=\d x^2+\d y^2$. The weight is a power $\gamma\in (-1,1)$ of the function 
\begin{align}
    d(z) \coloneq 1-|z|^2
    \label{eq:bdf}    
\end{align}
which is boundary-defining for $\Dm$ (i.e., $d|_{\partial\Dm} = 0$ with non-vanishing differential there). Over functions (i.e., zero-th order tensors), such a transform takes the well-known expression, in fan-beam coordinates $(\beta,\alpha) \in (\Rm/2\pi\Zm)\times [-\pi/2,\pi/2]$,
\begin{align}
I_0 d^\gamma f(z,v) &\coloneq \int_0^{2\cos\alpha} (d^\gamma f)(e^{i\beta} + te^{i(\beta+\alpha+\pi)})\ \d t, \quad f\in C^\infty(\Dm),
\label{eq:I0dgamma}
\end{align}
where we identify $(x,y)\in \Dm$ with the complex number $z=x+iy$. 

Transform \eqref{eq:I0dgamma} has applications in tomography, where the integrand may be supported all the way to the boundary, and where boundary behavior needs to be accounted for, whether singular or finite order of vanishing. As discussed in the recent review article on the topic \cite{monard2023non}, the study of boundary mapping properties for the X-ray transform, and the design of functional settings (Fr\'echet and Sobolev) where normal operators ($I_0$ composed with some weighted adjoint) can be made {\em invertible}, is fairly recent and was initiated as an important stepping stone toward the design of robust statistical estimators when dealing with noisy X-ray data \cite{Monard2017}. Known results leverage a combination of microlocal and singular analysis arguments \cite{Monard2017,Mazzeo2021} which hold on more general Riemannian manifolds, and specific Fourier-based model examples \cite{Monard2019a,Mishra2022,Monard2021} tailored to particular geometries such as simple geodesic disks of constant curvature. Transform \eqref{eq:I0dgamma} has a well-studied SVD \cite{Louis1984,hansen2021computed} and it was proved recently in \cite{Mishra2022} that the normal operator $(I_0 d^\gamma)^* I_0 d^\gamma$ (relative to an appropriate $L^2{-}L^2$ setting discussed below) is an isomorphism of $C^\infty(\Dm)$, that is tame with tame inverse. Tameness here is thought relative to a tame Fr\'echet structure on $C^\infty (\Dm)$ defined by some non-standard Sobolev semi-norms $\{\widetilde{H}^{s,\gamma}(\Dm)\}_{s\ge 0}$ coming from domain spaces of a degenerately elliptic operator ${\mathcal{L}}_\gamma$, see \cite{Mishra2022,monard2024boundary}. 

While much of this literature is focused on the X-ray transform over {\em functions}, a similar study over {\em tensor fields} needs to be made, for its applications to travel-time tomography \cite{Uhlmann2017}, Doppler tomography \cite{Holman2009} and inverse problems in transport \cite{Monard2017c,Fujiwara2019}, among others. Such a study should include defining the singularly-weighted X-ray transform on tensor fields, identifying its kernel, its mapping properties and its range. For tensor fields of order $m\ge 1$, since the transform has a non-trivial kernel, the design of gauge representatives modulo kernel which can be reconstructed from X-ray data, along with efficient reconstruction, is in order. In this context, prior works on the Euclidean disk appeared in the case $\gamma=0$ in \cite{Kazantsev2004,Monard2015a}, the former focused on the reconstruction of the solenoidal representative, the latter focused on an ``iterated-tt'' representative, where ``tt" stands for {\em transverse-tracefree}, constructed out of an iteration of a potential/conformal/tt decomposition result for tensor fields (see also \cite[Theorem 1.5]{Dairbekov2011}). A similar decomposition was also recently established for simply connected asymptotically hyperbolic surfaces in \cite{eptaminitakis2025tensor}, along with range characterization results and reconstruction procedures for even-order tensor fields in the case of the Poincar\'e disk.

The present article combines the philosophy of \cite{Monard2015a} in deriving an efficient approach to the tensor tomography problem, and the results known for the singularly-weighted X-ray transform on functions \cite{Mishra2022}. The main results are a thorough study of the $d^\gamma$-weighted X-ray transform on tensor fields, consisting of sharp kernel and range characterizations (Theorem \ref{thm:rangechar}), the design of an appropriate tensor representative modulo kernel (Corollary \ref{cor:decomposition}), and reconstruction procedures (Theorem \ref{thm:reconstruction}). The latter are flexible relative to tensor order in the sense that, unlike the problem of reconstructing the {\em solenoidal} representative of an $m$-tensor (see e.g. \cite{Kazantsev2004,Sharafudtinov1994}), whose technicalities strongly depend on $m$, the present approach proposes a rather simple and order-blind reconstruction approach. Another interesting feature that appears is the design of non-standard Sobolev spaces $H_0^{1,\gamma}(\Dm)$ (different from the ones from \cite{Mishra2022} discussed above), and new elliptic decompositions for the Guillemin-Kazhdan operators $\eta_\pm$ \cite{Guillemin1980} in weighted spaces, which help capture fine regularity properties of geodesic transport phenomena in this weighted context. 

We now state the main results and give an outline of the remainder at the end of the next section.

\section{Statement of main results}\label{sec:mainresults}

\subsection{Preliminaries and notation}\label{sec:introprelim}

Consider the unit disk $\Dm =\{(x,y)\in\Rm^2 \;|\; x^2+y^2\le 1\}$ with Euclidean metric $g= \d x^2+\d y^2$ and holomorphic coordinate $z=x+iy$. Here and below, we will denote the Wirtinger derivatives
\[ \partial := \frac{1}{2} (\partial_x - i\partial_y), \qquad \dbar := \frac{1}{2} (\partial_x + i\partial_y).  \]
The unit tangent bundle $\SD$ is parameterized by 
\begin{align}
    \Dm\times (\Rm/2\pi\Zm) \ni (z=x+iy,\theta) \mapsto (z,\cos\theta\partial_x + \sin\theta\partial_y) \in S\Dm,
    \label{eq:SDchart}
\end{align}
where the Sasaki volume form on $S\Dm$ equals $\d\Sigma^3:= \d x\d y\d\theta$. The inward-pointing boundary of $S\Dm$, $\partial_+ S\Dm$, as a model for all oriented geodesic segments through $\Dm$, is parameterized in {\em fan-beam coordinates} $(\Rm/2\pi\Zm)_\beta\times[-\frac{\pi}{2},\frac{\pi}{2}]_\alpha$, where $\beta$ parameterizes a boundary point $z=e^{i\beta}$ and $\alpha$ parameterizes the inward-pointing direction $v = -\cos(\beta+\alpha) \partial_x - \sin(\beta+\alpha) \partial_y$. The Euclidean geodesic flow emanating from $\partial_+ S\Dm$ is then parameterized as 
\begin{align}
    \varphi_t(\beta,\alpha) = ( z_{\beta,\alpha}(t) \coloneq e^{i\beta} + t e^{i(\beta+\pi+\alpha)}, \theta = \beta+\pi+\alpha) \in S\Dm, \quad (\beta,\alpha)\in \partial_+ S\Dm, \quad t\in [0,2\cos\alpha].
    \label{eq:geoflow}
\end{align}
The length of the geodesic $z_{\beta,\alpha}$ is $\tau(\beta,\alpha)=2\cos\alpha$ and with $d$ defined in \eqref{eq:bdf}, we have
\begin{align}
    d(z_{\beta,\alpha}(t))=(2\cos \alpha)t-t^2, \quad (\beta,\alpha)\in \inward,\quad t\in [0,2\cos\alpha].
    \label{eq:d_along_flow}
\end{align}
Definition \eqref{eq:I0dgamma} can then generalized to integrands on $S\Dm$, to define the operator $Id^\gamma \colon C^\infty(\SD) \to C^\infty(\inward)$ given by 
\begin{align}
    I d^\gamma f(\beta,\alpha) \coloneq \int_0^{2\cos\alpha} f(\varphi_t(\beta,\alpha))d^\gamma(z_{\beta,\alpha}(t)) \d t, \quad f\in C^\infty(S\Dm).
    \label{eq:Idgamma}
\end{align}

\noindent{\bf Notation.} In what follows, several weighted Hilbert spaces will appear, and we will adopt the following shorthand notation: $L^2_\gamma(\Dm)$ will denote $L^2(\Dm,d^\gamma|\d z|^2)$; $L^2_\gamma(S\Dm)$ will denote $L^2(\SD, d^\gamma \d\Sigma^3)$ and $L^2_\gamma (\inward)$ will denote $L^2(\inward, \mu^{-2\gamma}\d\beta \d\alpha)$. The space $L^2_\gamma (\inward)$ also admits an orthogonal splitting relative to the orientation-reversing\footnote{also referred to as ``antipodal scattering relation"} map $\mathcal{S}_A\colon \partial_+ S\Dm \to \partial_+ S\Dm$
\begin{align}
    (\beta,\alpha) \mapsto (\beta+\pi+2\alpha, -\alpha), \qquad (\beta,\alpha)\in \partial_+ S\Dm.
\end{align}
Such a map preserves the measure $\d\beta\d\alpha$, and we thus have the orthogonal splitting 
\begin{align}
\begin{split}
    L^2_\gamma(\inward) &= L^2_{\gamma,-}(\inward) \stackrel{\perp}{\oplus} L^2_{\gamma,+}(\inward), \\
    \text{where}\quad L^2_{\gamma,\pm} (\partial_+ S\Dm) &:= L^2_{\gamma} (\partial_+ S\Dm) \cap \ker (\Id \mp \mathcal{S}_A^*).     
\end{split}
\label{eq:ortho}
\end{align}
For $m\ge 1$, we will denote $L^2_\gamma(\Dm; S^m(T^*\Dm))$ the space of symmetric tensors of order $m$ on $\Dm$, all of whose components belong to $L^2_\gamma(\Dm)$.

We will denote by $\star$ the Hodge star operator on one-forms, characterized by the relations
\begin{align}
    \star\d x = \d y, \quad \star\d y= - \d x.
    \label{eq:Hodge}
\end{align}

Finally, we denote $L^2_\gamma(\Dm)\cap \ker \dbar$ and $L^2_\gamma(\Dm)\cap \ker \partial$ the closed subspaces of $L^2_\gamma(\Dm)$ made of $L^2_\gamma(\Dm)$-integrable analytic and antianalytic functions on $\Dm^\circ$. Such spaces are also characterized in Lemma \ref{lem:charholomorphic} below.

\subsection{Forward mapping properties}

It was shown in \cite{Mishra2022} that the operator
$I_0d^\gamma : L_\gamma^2(\Dm) \to L_{\gamma,+}^2(\inward)$ is bounded and with infinite-dimensional co-kernel. Here we extend this analysis to the case of integrands defined on $S\Dm$, and show that this extension is continuous and surjective. Here and below, we denote the Beta function
\begin{align}
    B(x,y):= \int_0^1 s^{x-1}(1-s)^{y-1}\d s, \quad x,y>0.
    \label{eq:beta}
\end{align}

\begin{proposition}\label{prop:Ibounded}
    For every $\gamma>-1$, the operator
    \begin{align}
        Id^\gamma \colon L^2_\gamma(\SD) \to L^2_\gamma(\inward)
    \end{align}
    is bounded (with norm at most $2^{\gamma+1/2} \sqrt{B(\gamma+1,\gamma+1)}$) and surjective.    
\end{proposition}

Below, rank-$m$ symmetric covariant tensor fields on $\Dm$ will be identified with certain functions on $\SD$ via the map
\begin{align}
    \ell_m \colon C^\infty(\Dm; S^m(T^*\Dm)) \to C^\infty(\SD), \qquad \ell_m f(z,v) \coloneq f_z(v,\dots,v),
    \label{eq:ellm}
\end{align}
where the argument $v$ of $f_z$ is repeated $m$ times. The map extends continuously to the setting $\ell_m \colon L^2_\gamma(\Dm; S^m(T^*\Dm)) \to L^2_\gamma(\SD)$. Combined with Proposition \ref{prop:Ibounded}, this gives the following.

\begin{corollary}\label{cor:Imbounded}
The integral operator $I_m d^\gamma \colon L^2_\gamma(\Dm; S^m(T^*\Dm))
	\to L^2_\gamma(\inward)$
defined by
\begin{align}
I_m d^\gamma f(\beta,\alpha) &\coloneq \int_0^{2\cos\alpha}
	\ell_m f(\gamma_{\beta,\alpha}(t),\dot\gamma_{\beta,\alpha}(t))d^\gamma(\gamma_{\beta,\alpha}(t))\d t, \qquad (\beta,\alpha)\in\inward
\end{align}
is bounded.
\end{corollary}

\subsection{Potential tensors}
We first discuss the appropriate description of the natural kernel of $I_m d^\gamma$. Let $H_0^{1,\gamma}(\Dm)$ denote the completion of $C_c^\infty(\Dm^\circ)$ with respect to the norm
\begin{align}
\|u\|_{H^{1,\gamma}(\Dm)}^2 \coloneq \|u\|_{L_\gamma^2(\Dm)}^2 + \|\nabla u\|_{L_\gamma^2(\Dm)}^2.
\label{eq:H01gamma}
\end{align}
For $m\ge 1$, define $H_0^{1,\gamma}(\Dm; S^m(T^* \Dm))$ the symmetric $m$-tensor fields all of whose components lie in $H_0^{1,\gamma}(\Dm)$. Recall the definition of the {\em inner derivative}
\begin{align}
\d^s = \sigma\circ\nabla \colon C^\infty(\Dm;S^m(T^*\Dm)) \to C^\infty(\Dm;S^{m+1}(T^*\Dm)),
\label{eq:ds}
\end{align}
where $\nabla\colon C^\infty(\Dm;S^m(T^*\Dm)) \to C^\infty(\Dm;S^{m+1}(T^*\Dm))$ denotes the total covariant derivative and $\sigma$ is the symmetrization operator. Then it is easy to establish that for any $m\ge 0$ and $\gamma\in (-1,1)$, the operator $\d^s$ extends to a bounded operator $\d^s \colon H_0^{1,\gamma}(\Dm; S^m(T^* \Dm)) \to L^2_\gamma (\Dm; S^{m+1}(T^* \Dm))$. Moreover, we have
\begin{lemma}\label{lem:kernel}
If $\gamma\in (-1,1)$, $m\ge 0$, and $q\in H^{1,-\gamma}_0(\Dm; S^m(T^*\Dm))$, then it follow that 
\[ I_m d^\gamma (d^{-\gamma}\d^s q)=0. \]
\end{lemma}
\begin{proof}[Proof of Lemma \ref{lem:kernel}]
The result is well-known for $q\in C_c^\infty (\Dm^\circ,S^m (T^* \Dm^\circ))$ (see, e.g., \cite[Theorem 3.3.2]{Sharafudtinov1994}) and extends by density through the composition of bounded operators
\begin{align*}
    H^{1,-\gamma}_0(\Dm; S^m(T^*\Dm)) \xrightarrow{d^{-\gamma} \d^s}L^2_\gamma (\Dm; S^{m+1}(T^* \Dm)) \xrightarrow{I_m d^\gamma} L^2_\gamma (\partial_+ S\Dm).
\end{align*}
\end{proof}

\subsection{Tensor field decompositions}

On tensor fields we also define the symmetric product with the metric 
\begin{align}
    L \colon C^\infty(\Dm;S^m(T^*\Dm)) \to C^\infty(\Dm;S^{m+2}(T^*\Dm)), \qquad Lv \coloneq \sigma (g\otimes v),
    \label{eq:Lmap}    
\end{align}
with adjoint the {\em trace} map (e.g., with respect to the last two indices)
\begin{align}
\tr \colon C^\infty(\Dm;S^m(T^*\Dm)) \to C^\infty(\Dm;S^{m-2}(T^*\Dm)), \qquad m\ge 2,
\end{align}
the {\em divergence} $\delta \colon C^\infty(S^m(T^*\Dm)) \to C^\infty(S^{m-1}(T^*\Dm))$ as the formal adjoint of $-\d^s$ defined in \eqref{eq:ds}, which also takes the expression $\delta = \tr \nabla$. For an integer $m\ge 1$, let us define 
\begin{align}
    L^2_\gamma (\Dm; \Stt^m (T^* \Dm)) \coloneq \left\{f_m \d z^m + f_{-m} \d\zbar^m \;|\; f_m \in L^2_\gamma (\Dm)\cap \ker \dbar,\ f_{-m} \in L^2_\gamma (\Dm)\cap \ker\partial\right\}.
    \label{eq:Stt}
\end{align}
For $m\ge 2$, it can be seen that 
\begin{align}
    L^2_\gamma (\Dm; \Stt^m (T^* \Dm)) = \left\{ f\in L^2_\gamma (\Dm; S^m (T^* \Dm)),\ \delta f = 0,\ \tr f = 0\right\},
    \label{eq:Stt_correspondence} 
\end{align}
justifying the subscript ``tt'' which stands for ``transverse-tracefree''. We extend this notation to $m=1$, in a small abuse of notation, since 1-tensors do not have a well-defined trace, although this space can also be viewed as the space of $L^2_\gamma$-integrable harmonic one-forms. 

We begin with two decompositions of one-forms. 

\begin{lemma}\label{lem:oneformdecompSD}
(1) Any $w\in L^2_\gamma(\Dm; S^1(T^*\Dm))$ admits a unique decomposition of the form
\begin{align}
w = d^{-\gamma} \d g_0 + d^{-\gamma} {\star \d} g_s + \tilde{g}_1
\label{eq:oneform1}
\end{align}
where $g_0,g_s\in H^{1,-\gamma}_0(\mathbb{D})$
and $\tilde{g}_1 \in L^2(\Dm, \Stt^1(T^*\Dm))$.

(2) Any $w\in L^2_\gamma(\Dm; S^1(T^*\Dm))$ admits a unique decomposition of the form
\begin{align}
    w = d^{-\gamma} \d f + \tilde{w}_{-1}\d\zbar + \tilde{w}_1 \d z
    \label{eq:oneform2}
\end{align}
where $f\in H^{1,-\gamma}_0(\mathbb{D})$, $\tilde{w}_{1} \in L^2_\gamma(\Dm)$ and $\tilde{w}_{-1} \in L^2_\gamma(\Dm) \cap \ker \partial$.
\end{lemma}

Moving to higher-order tensor fields, the decomposition below is a generalization of the decomposition presented in \cite[Theorem 1.5]{Dairbekov2011} to tensor fields of more general smoothness classes at the boundary. 
\begin{theorem}[Tensor field decomposition] \label{thm:Yorkdecomp}
Fix $\gamma\in (-1,1)$ and $m\ge 2$. Then any tensor $f\in L^2_\gamma(\Dm; S^m(T^* \Dm))$ decomposes uniquely in the form
\begin{align}
    f = d^{-\gamma} \d^s q + L \lambda + \tilde{f},
    \label{eq:York}
\end{align}
where $q\in H^{1,-\gamma}_0(\Dm, S^{m-1}(T^* \Dm)) \cap \ker \tr$,
$\lambda\in L^2_\gamma(\Dm; S^{m-2}(T^* \Dm))$ and
$\tilde{f}\in L^2_\gamma(\Dm, \Stt^m(T^* \Dm))$.
\end{theorem}

By iterating Theorem \ref{thm:Yorkdecomp} and using decomposition \eqref{eq:oneform1} as base case when $m=1$, tensor fields of any order can be decomposed as follows.

\begin{corollary}\label{cor:decomposition} Fix $m\ge 1$. Then any $f\in L^2_\gamma(\Dm; S^{m}(T^*\Dm))$ uniquely decomposes in the form 
\[  f = d^{-\gamma}\d^s q + \fitt, \qquad q\in H^{1,-\gamma}_0(\Dm; S^{m-1}(T^*\Dm)),  \]
and where $\fitt\in L^2_\gamma(\Dm; S^{m}(T^*\Dm))$ takes the following form: 
\begin{itemize}
    \item If $m = 2p$ where $p>0$, then for $1\le j\le p$,
    \begin{align}
        \fitt = \sum_{j=0}^p L^{p-j} \tilde{f}_{2j}, \quad \tilde{f}_0\in L^2_\gamma(\Dm),\ \tilde{f}_{2j}\in L^2_\gamma(\Dm; \Stt^{2j}(T^*\Dm))
        \label{eq:fitt_even}
    \end{align}
    Moreover, $I_{2p}d^\gamma f = I_{2p} d^\gamma \fitt = \sum_{j=0}^p I_{2j} d^\gamma \tilde{f}_{2j}$.
    \item If $m=2p+1$ where $p\ge 0$, then for $0\le j\le p$,
    \begin{align}
        \fitt = d^{-\gamma}{\star\d} h + \sum_{j=0}^{p}L^{p-j} \tilde{f}_{2j+1}, \quad h\in H^{1,-\gamma}_0(\Dm),\ \tilde{f}_{2j+1}\in L^2_\gamma(\Dm; \Stt^{2j+1}(T^*\Dm)).
        \label{eq:fitt_odd}
    \end{align}
    Moreover, $I_{2p+1}d^\gamma f = I_{2p+1} d^\gamma \fitt = I_1 (\star \d h) + \sum_{j=0}^p I_{2j+1} d^\gamma \tilde{f}_{2j+1}$.   
\end{itemize}
\end{corollary}
We say an $m$-tensor $f\in L^2_\gamma(\Dm;S^m(T^*\Dm))$ is {\em in the iterated-tt gauge} if $f$ is as in the previous corollary and $q=0$.

\subsection{Data space decomposition}
Now we discuss where in data space $L^2_\gamma(\inward)$ the $d^\gamma$-weighted X-ray transform sends components of the decomposition in Corollary \ref{cor:decomposition}. First observe that, relative to the orthogonal splitting \eqref{eq:ortho}, the range of $I_m d^\gamma$ is a subset of $L^2_{\gamma,+}(\inward)$ or $L^2_{\gamma,-}(\inward)$ if $m$ is even or odd, respectively.

From \cite{Mishra2022}, an orthonormal basis of $L^2_{\gamma,+}(\inward)$ is given by the weighted \emph{fan-beam polynomials}
\begin{align}
&\psi_{n,k}^\gamma \coloneq \mu^{2\gamma+1}e^{(n-2k)i(\beta+\alpha+\pi)} 
\hat L_n^\gamma(\sin\alpha)/2\pi, \qquad n\ge 0,\ k\in \Zm, \label{def:psink}
\end{align}
where $\hat L_n^\gamma\colon [-1,1]\to \Rm$ is the orthogonal polynomial of degree $n$ with respect to the weight $(1-x^2)^{\gamma+1/2}$. In Appendix \ref{sec:normalization}, a normalization is chosen for $\hat L_n^\gamma$ such that $\|\psi_{n,k}^\gamma\|_{L^2_\gamma(\inward)}=1$. The isometry 
\[ L^2_{\gamma,+}(\inward) \to L^2_{\gamma,-}(\inward), \quad u\mapsto e^{i(\beta+\alpha+\pi)}u  \]
shows that $\{e^{i(\beta+\alpha+\pi)} \psi_{n,k}^\gamma\}_{n\ge 0, k\in \Zm}$ is an orthonormal basis of $L^2_{\gamma,-}(\inward)$. Following \cite[Equation (4.34)]{BohrNickl2021}, it is convenient to define the basis functions 
\begin{align}
    \psi_{n,k}^{\gamma,+} \coloneq \psi_{n,k}^\gamma, \qquad \psi_{n,k}^{\gamma,-} \coloneq e^{i(\beta+\alpha+\pi)}\psi_{n,k}^\gamma, \quad n\ge 0,\ k\in \Zm.
    \label{eq:psinkpm}    
\end{align}

Mapping properties of $I_0d^\gamma$ and estimates for its spectral action from \cite[Theorem 1 and Section 5.2]{Mishra2022} indicate that
\begin{align}
    \begin{split}
        I_0d^\gamma(L^2_\gamma(\Dm)) &= \left\{ \sum_{n=0}^\infty \sum_{k=0}^n \sigma_{n,k}^\gamma a_{n,k} \psi_{n,k}^{\gamma,+}, \quad \sum_{n=0}^\infty \sum_{k=0}^n |a_{n,k}|^2 <\infty \right\}, \\
        \text{where}\quad (\sigma_{n,k}^\gamma)^2 &:= 2^{2\gamma+2}\pi\binom{n}{k}\frac{(n-k+\gamma)! (k+\gamma)!}{(n+2\gamma+1)!},
    \end{split}    
    \label{eq:I0range}
\end{align}
where for $x>-1$, $x!:= \Gamma(x+1)$ and where $\Gamma$ denotes Euler's Gamma function. Such singular values all tend to zero as $n\to \infty$ (see, e.g., \cite[Sec. 5.2]{Mishra2022}), hence taking the $L^2_\gamma$-closure, we obtain the subspace of $L^2_{\gamma,+}(\inward)$
\begin{align}
\overline{I_0 d^\gamma (L_\gamma^2(\Dm))}^{L^2_\gamma(\inward)} &= \text{span}_{L^2_\gamma} \left\{ \psi_{n,k}^{\gamma,+}\;|\; n\ge 0,\ 0\le k\le n\right\}. 
\label{eq:I0rangeclosure}
\end{align}

To describe the higher-order components of $d^\gamma$-weighted X-ray data, define the following. 
\begin{definition}[Hilbert Scales]
    For a Hilbert space $H$, an orthonormal family $\{f_{n}\}_{n\ge 0}$ in $H$ and $\alpha\in [0,\infty)$, define
    \begin{align}
        h^\alpha (f_{n},\ n\ge 0) \coloneq \left\{ \sum_{n=0}^\infty a_{n} f_{n} \;|\; \sum_{n=0}^\infty (n+1)^{2\alpha} |a_{n}|^2 < \infty \right\}.
    \end{align}
    If $\alpha = 0$, we write $h^0 = \ell^2$.
\end{definition}

\begin{lemma}\label{lem:homeomorphisms}
Fix $\gamma>-1$. Then the following operators are homeomorphisms: 
\begin{align}
    \begin{split}
        I_{2p} d^\gamma &\colon L^2_\gamma(\Dm, \Stt^{2p}(T^*\Dm)) \to h^{(1+\gamma)/2}(\psi_{n,-p}^{\gamma,+}, n\ge 0) \oplus h^{(1+\gamma)/2}(\psi_{n,n+p}^{\gamma,+}, n\ge 0), \quad p\ge 1, \\
        I_{2p+1} d^\gamma &\colon L^2_\gamma(\Dm, \Stt^{2p+1}(T^*\Dm)) \to h^{(1+\gamma)/2} (\psi_{n,-p}^{\gamma,-}, n\ge 0) \oplus h^{(1+\gamma)/2} (\psi_{n,n+p+1}^{\gamma,-}, n\ge 0), \quad p\ge 0.     
    \end{split}
    \label{eq:homeos}   
\end{align}
\end{lemma}

As $p$ varies, the ranges above span over disjoint sets of fan-beam polynomial indices, showing their orthogonality. Together with \eqref{eq:I0range} and \eqref{eq:I1span} below, this induces an orthogonal decomposition of the range of $I_md^\gamma$ for any $m$.
\begin{theorem}[Orthogonal decomposition of the range]\label{thm:rangedecomp}
Let $m$ be a positive integer. If $m=2p$ for $p\ge 1$, 
then the range of $I_m d^\gamma$ decomposes orthogonally
\begin{align}
    I_md^\gamma(L^2_\gamma(\Dm; S^m(T^* \Dm))) = I_0 d^\gamma (L^2_\gamma(\Dm)) \stackrel{\perp}{\oplus} \bigoplus_{j=1}^{p} I_{2j} d^\gamma (L^2_\gamma(\Dm; \Stt^{2j}(T^*\Dm))).
    \label{eq:evendecomposition}
\end{align}
If $m=2p+1$ for $p\ge 0$, then the range of $I_md^\gamma$ decomposes orthogonally
\begin{align}
    I_m d^\gamma(L^2_\gamma(\Dm; S^m(T^* \Dm))) = I_1(\star\d ( H_0^{1,-\gamma}(\Dm))) \stackrel{\perp}{\oplus} \bigoplus_{j=0}^{p} I_{2j+1}d^\gamma (L^2_\gamma(\Dm; \Stt^{2j+1}(T^*\Dm))),
    \label{eq:odddecomposition}
\end{align}
where 
\begin{align}
    \overline{I_1(\star\d ( H_0^{1,-\gamma}(\Dm)))}^{L^2_\gamma(\inward)} = \operatorname{span}_{L^2} \{ \psi_{n,k}^{\gamma,-} \;|\; n\geq 1, 1\leq k\leq n \}.
    \label{eq:I1span}
\end{align}
Figures \ref{fig:Ievenrange} and \ref{fig:Ioddrange} depict the decomposition.
\end{theorem}

\newcommand{\gridpoint}[2]{({#2-(#1)/2}, {(#1)*sqrt(3)/2})}
\newcommand{\gridlabel}[2]{({#2-(#1)/2}, {(#1)*sqrt(3)/2-0.3})}
\newcommand\N{3}
\newcommand\boxscale{0.6}
\newcommand\figscale{1.15}
\begin{figure}[h]
\centering
\scalebox{\boxscale}{
\begin{tikzpicture}[x=1cm,y=1cm,scale=\figscale]
\newcommand\arrowlabel{$\xrightarrow{\displaystyle I(d^\gamma \wtZ_{n,k}^\gamma e^{2j i\theta})}$}

\path\gridlabel{\N+1}{-\N+2.5} node(x) {$j>0$};
\path\gridlabel{\N+1}{\N/2+0.5} node(x) {$j=0$};
\path\gridlabel{\N+1}{2*\N-1.5} node(x) {$j<0$};

\foreach \k [parse=true] in {1,...,\N} {
    \fill [cyan,opacity=0.125] \gridpoint{0}{0.5} -- \gridpoint{\N}{\N+0.5}
        -- \gridpoint{\N}{-\k-0.5} -- \gridpoint{0}{-\k-0.5};
}
\foreach \k [parse=true] in {1,...,\N} {
    \fill [magenta,opacity=0.125] \gridpoint{0}{-0.5} -- \gridpoint{\N}{-0.5}
        -- \gridpoint{\N}{\N+\k+0.5} -- \gridpoint{0}{\k+0.5};
}

\draw[->,dashed] \gridpoint{0}{0.5} -- \gridpoint{\N+0.25}{\N+0.75};
\draw[->] \gridpoint{0}{-0.5} -- \gridpoint{\N+0.25}{-0.5};

\draw[->] \gridpoint{0}{-(\N+mod(\N,2))}--\gridpoint{0}{\N+mod(\N,2)};

\foreach \k [parse=true] in {1,...,\N} {
    \pgfmathtruncatemacro\kminus{-\k}
    \foreach \n in {0,...,\N} {
        \draw[fill] \gridpoint{\n}{-\k} + (-1pt,0)
            arc[start angle=180,end angle=360,radius=1pt];
        \draw \gridpoint{\n}{-\k} circle (1pt);
        \path \gridlabel{\n}{-\k} node(x) {$\psi_{\n,\kminus}^{\gamma,+}$};

        \pgfmathtruncatemacro\kplus {\k+\n}
        \draw[fill] \gridpoint{\n}{\kplus} + (1pt,0)
            arc[start angle=0,end angle=180,radius=1pt];
        \draw \gridpoint{\n}{\kplus} circle (1pt);
        \path \gridlabel{\n}{\kplus} node(x) {$\psi_{\n,\kplus}^{\gamma,+}$};
    }
}

\foreach \n in {0,...,\N} {
    \foreach \k in {0,...,\n} {

        \fill \gridpoint{\n}{\k} circle (1pt);
        \path \gridlabel{\n}{\k} node(x) {$\psi_{\n,\k}^{\gamma,+}$};
    }
}
\end{tikzpicture}}

\caption{Decomposition of $L^2_{\gamma,+}(\partial_+ S\Dm)$ induced by Equation \eqref{eq:evendecomposition}: $\operatorname{Ran}(I_0d^\gamma)$ is spanned by the polynomials in-between the solid and dashed lines. For fixed $j\geq 1$, the diagonals $\left\{\psi_{n,-j}^{\gamma,+}\right\}$ and $\left\{\psi_{n,n+j}^{\gamma,+}\right\}$ capture the range of $L^2_\gamma(\Dm;\Stt^{2j}(T^*\Dm))$.}
\label{fig:Ievenrange}
\end{figure}

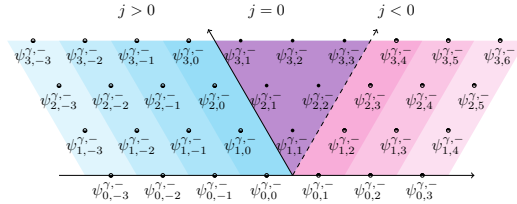
\begin{figure}[h]
\centering
\scalebox{\boxscale}{
\begin{tikzpicture}[x=1cm,y=1cm,scale=\figscale]
\newcommand\arrowlabel{$\xrightarrow{\displaystyle I(d^\gamma \wtZ_{n,k}^\gamma e^{(2j+1)i\theta})}$}

\path\gridlabel{\N+1}{-\N+2.5} node(x) {$j>0$};
\path\gridlabel{\N+1}{\N/2+0.5} node(x) {$j=0$};
\path\gridlabel{\N+1}{2*\N-1.5} node(x) {$j<0$};

\foreach \k [parse=true] in {0,...,\N} {
    \fill [cyan,opacity=0.125] \gridpoint{0}{0.5} -- \gridpoint{\N}{0.5}
        -- \gridpoint{\N}{-\k-0.5} -- \gridpoint{0}{-\k-0.5};
}
\foreach \k [parse=true] in {0,...,\N} {
    \fill [magenta,opacity=0.125] \gridpoint{0}{0.5} -- \gridpoint{\N}{\N+0.5}
        -- \gridpoint{\N}{\N+\k+0.5} -- \gridpoint{0}{\k+0.5};
}

\fill [cyan,opacity=0.375]
    \gridpoint{0}{0.5} -- \gridpoint{3}{3.5} -- \gridpoint{3}{0.5} -- \gridpoint{0}{0.5};
\fill [magenta,opacity=0.375]
    \gridpoint{0}{0.5} -- \gridpoint{3}{3.5} -- \gridpoint{3}{0.5} -- \gridpoint{0}{0.5};

\draw[->,dashed] \gridpoint{0}{0.5} -- \gridpoint{\N+0.25}{\N+0.75};
\draw[->] \gridpoint{0}{0.5} -- \gridpoint{\N+0.25}{0.5};

\draw[->] \gridpoint{0}{-(\N+mod(\N,2))}--\gridpoint{0}{\N+mod(\N,2)};

\foreach \k in {0,...,\N} {
    \pgfmathtruncatemacro\kminus{-\k}
    \foreach \n in {0,...,\N} {
        \draw[fill] \gridpoint{\n}{-\k} + (-1pt,0)
            arc[start angle=180,end angle=360,radius=1pt];
        \draw \gridpoint{\n}{-\k} circle (1pt);
        \path \gridlabel{\n}{-\k} node(x) {$\psi_{\n,\kminus}^{\gamma,-}$};
    }
}

\foreach \k in {1,...,\N} {
    \pgfmathtruncatemacro\kminus{-\k}
    \foreach \n in {0,...,\N} {
        \pgfmathtruncatemacro\kplus {\k+\n}
        \draw[fill] \gridpoint{\n}{\kplus} + (1pt,0)
            arc[start angle=0,end angle=180,radius=1pt];
        \draw \gridpoint{\n}{\kplus} circle (1pt);
        \path \gridlabel{\n}{\kplus} node(x) {$\psi_{\n,\kplus}^{\gamma,-}$};
    }
}

\foreach \n in {1,...,\N} {
    \foreach \k in {1,...,\n} {
        \fill \gridpoint{\n}{\k} circle (1pt);
        \path \gridlabel{\n}{\k} node(x) {$\psi_{\n,\k}^{\gamma,-}$};
    }
}
\end{tikzpicture}}

\caption{Decomposition of $L^2_{\gamma,-}(\partial_+S\Dm)$ induced by Equation \eqref{eq:odddecomposition}: $\operatorname{Ran}(I_1(\star \d))$ is spanned by the polynomials in between the solid and and dashed lines. For fixed $j\geq 0$, the diagonals $\left\{\psi_{n,-j}^{\gamma,-}\right\}$ and $\left\{\psi_{n,n+j+1}^{\gamma,-}\right\}$ capture the range of $L^2_\gamma(\Dm;\Stt^{2j+1}(T^*\Dm))$.}
\label{fig:Ioddrange}
\end{figure}

\subsection{Range characterization and reconstruction}
Decompose the identity map $\Id|_{L^2_{\gamma,+}(\inward)}$ so that $\Id = \Pi_0 + \Pi_2 + \cdots + \Pi_{2j} + \cdots$, where
\begin{align}
\Pi_0 &\colon L^2_{\gamma,+}(\inward) \to
    \text{span}_{L^2_\gamma}(\psi_{n,k}^{\gamma,+}, n\ge 0, 0 \leq k\leq n) \quad\text{and}\\
\Pi_{2j} &\colon L^2_{\gamma,+}(\inward) \to \ell^2(\psi_{n,-j}^{\gamma,+}, n\ge 0)
        \oplus \ell^2(\psi_{n,n+j}^{\gamma,+}, n\ge 0).
\end{align}
for $j\geq 1$. Likewise, decompose the identity map
$\Id_{L^2_{\gamma,-}(\inward)}$
into $\Id = \Pi_\perp + \Pi_1 + \Pi_3 + \cdots + \Pi_{2j+1} + \cdots$, where 
\begin{align}
\begin{aligned}
\Pi_\perp &\colon L^2_{\gamma,-}(\inward) \to \text{span}_{L^2_\gamma}(\psi_{n,k}^{\gamma,-}, n\geq 1, 1\leq k\leq n) \quad\text{and}\\
\Pi_{2j+1} &\colon L^2_{\gamma,-}(\inward) \to  \ell^2(\psi_{n,-j}^{\gamma,-}, n\ge 0) \oplus \ell^2 (\psi_{n,n+j+1}^{\gamma,-}, n\ge 0), \quad j\ge 0.
\end{aligned}
\end{align}

The next theorem provides a characterization for whether an element of $L^2_{\gamma,\pm}(\inward)$ equals $I_md^\gamma g$ for some integer $m$ and $m$-tensor $g$.

\begin{theorem}[Range characterization]\label{thm:rangechar}
Let $p\geq 0$ be an integer. A function $u\in L^2_{\gamma,+}(\inward)$ takes the form $u = I_{2p} d^\gamma f$ for some $f\in L^2_\gamma(\Dm, S^{2p}(T^*\Dm))$ if and only if

(a) For all $j>p$, $\Pi_{2j} u =0$,

(b) For all $1\le j\le p$, $\Pi_{2j}u\in h^{(1+\gamma)/2}(\psi_{n,-j}^{\gamma,+}, n\ge 0) \oplus h^{(1+\gamma)/2}(\psi_{n,n+j}^{\gamma,+}, n\ge 0)$, and

(c) Writing $\Pi_0 u = \sum_{n=0}^\infty\sum_{k=0}^n u_{n,k}\psi_{n,k}^{\gamma,+}$, we have
\begin{align}
    \sum_{n=0}^\infty\sum_{k=0}^n \frac{|u_{n,k}|^2}{(\sigma_{n,k}^\gamma)^2} < \infty.
    \label{eq:spectraldecay_even}
\end{align}

A function $u\in L^2_{\gamma,-}(\inward)$ takes the form $u = I_{2p+1} d^\gamma f$ for some $f\in L^2_\gamma(\Dm, S^{2p+1}(T^*\Dm))$ if and only if

(a) For all $j>p$, $\Pi_{2j+1} u =0$,

(b) For all $0\le j\le p$, $\Pi_{2j+1}u\in h^{(1+\gamma)/2} (\psi_{n,-j}^{\gamma,-}, n\ge 0) \oplus h^{(1+\gamma)/2} (\psi_{n,n+j+1}^{\gamma,-}, n\ge 0)$, and

(c) Writing $\Pi_\perp u = \sum_{n=1}^\infty\sum_{k=1}^n u_{n,k}\psi_{n,k}^{\gamma,-}$, we have
\begin{align}
    \sum_{n=1}^\infty\sum_{k=1}^n \frac{|u_{n,k}|^2}{(\sigma_{n,k}^\gamma)^2} < \infty.
    \label{eq:spectraldecay_odd}
\end{align}
\end{theorem}

In the proof of the theorem above, when a data function $u$ satisfies the conditions to be in the range, the construction of a preimage is based on the range characterizations \eqref{eq:I0range} for $I_0 d^\gamma(L^2_\gamma(\Dm))$, \eqref{eq:rangeIperp} for $I_1 (\star (\d  H_0^{1,-\gamma}(\Dm)))$, as well as the homeomorphisms of Lemma \ref{lem:homeomorphisms} for the tt components. 

In the result below, we provide an alternate inversion formula for the tt components based on an integral kernel. 

\begin{theorem}[Reconstruction of the tt modes]\label{thm:reconstruction}
For each integer $j\geq 0$ and $\gamma\in(-1,1)$, define
\begin{align}
G_{2j}^\gamma(\beta,\alpha;z) &\coloneq \frac{\mu^{2\gamma+1}e^{2ji(\beta+\alpha+\pi)}}{2^{4\gamma+2}\gamma!^2} \frac{(\gamma+1)(1+e^{i(2\beta+2\alpha)}z^2)}{((e^{i\beta}z+1)(e^{i(\beta+2\alpha+\pi)}z+1))^{\gamma+2}}.
\label{eq:kernel}
\end{align}
The modes of a symmetric $m$-tensor in \emph{iterated tt-form} are reconstructed by setting, if $m=2p$:
\begin{align}
f_{2j} = (\mathcal{D}, G_{2j}^\gamma(\cdot,\cdot; z))_{L^2_\gamma}\ \d z^{2j}
+ (\mathcal{D}, \overline{G_{2j}^\gamma(\cdot,\cdot; z)})_{L^2_\gamma}\ \d \zbar^{2j},
\label{eq:reconeven}
\end{align}
for $1\leq j\leq p$ in any order. If $m=2p{+}1$, then set:
\begin{align}
f_{2j+1} = (\mathcal{D}, e^{i(\beta+\alpha+\pi)}G_{2j}^\gamma(\cdot,\cdot; z))_{L^2_\gamma}\ \d z^{2j+1}
+ (\mathcal{D}, \overline{e^{i(\beta+\alpha+\pi)}G_{2j}^\gamma(\cdot,\cdot;z)})_{L^2_\gamma}\ \d \zbar^{2j+1},
\label{eq:reconodd}
\end{align}
for $0\leq j\leq p$ in any order.
\end{theorem}

From the reconstruction Theorem \ref{thm:reconstruction} and further discussions in Section \ref{sec:reconstruction} on the unique reconstruction of $f\in L^2_\gamma(\Dm)$ from $I_0 d^\gamma f$ and of $h\in H^{1,-\gamma}_0(\Dm)$ from $I_1 (\star\d h)$, if $I_m d^\gamma f= 0$ for some $f\in L^2_\gamma(\Dm; S^m(T^*\Dm))$ in iterated-tt form, then $f=0$. Combining this with Corollary \ref{cor:decomposition}, we state without proof the following solenoidal injectivity result in a weighted context. 

\begin{corollary}[Solenoidal injectivity]\label{cor:sinjectivity} 
    If $I_m d^\gamma f = 0$ for some $f\in L^2_\gamma(\Dm, S^m(T^*\Dm))$, then f = $d^{-\gamma} \d^s q$ for some $q\in H^{1,\gamma}_0(\Dm; S^{m-1}(T^* \Dm))$.
\end{corollary}

\noindent{\bf Outline.} The remainder of the article is organized as follows. We address forward mapping properties in Section \ref{sec:forward}, including the proof of Proposition \ref{prop:Ibounded}. 
In Section \ref{sec:gauge}, we prove the results on tensor field decompositions. This first requires a preliminary elliptic decomposition result for a weighted $\dbar$ operator in Section \ref{sec:elldecunitdisk}, and some preliminaries on fiberwise Fourier analysis in Section \ref{sec:fibFourier}. We then prove Theorem \ref{thm:Yorkdecomp} and Corollary \ref{cor:decomposition} in Section \ref{sec:proofdecomp}.
In Section \ref{sec:rangeDecomp}, we address the data space decomposition in terms of the range of the X-ray transform over different tensor field components, covering notably the proofs of Lemma \ref{lem:homeomorphisms} and Theorem \ref{thm:rangedecomp}. 
In Section \ref{sec:rangeCharac}, we prove Theorem \ref{thm:rangechar} on the range characterization of $I_m d^\gamma$. 
In Section \ref{sec:reconstruction}, we address reconstruction aspects, including the proof of Theorem \ref{thm:reconstruction} and a discussion on the inversion of $I_0 d^\gamma$ and $I_1 {\star}\d$.
In Appendix \ref{sec:normalization}, we discuss our choice of eigenfunction normalizations, and in Appendix \ref{sec:poincare}, we prove a Poincar\'e type inequality that is used in Section \ref{sec:elldecunitdisk}.


\section{Forward mapping properties - Proof of Proposition \ref{prop:Ibounded}}\label{sec:forward}
The fundamental change-of-variables formula on $S\Dm$ is given by Santal\'o's formula, as follows.
\begin{align}
\int_\inward \int_0^{\tau(\beta,\alpha)} f(\varphi_t(\beta,\alpha)) \d t\ \cos\alpha\d\beta\d\alpha = \int_\SD f\ \d\Sigma^3, \qquad f\in C^\infty(S\Dm),
\label{eq:Santalo}
\end{align}
see, e.g. \cite[Proposition 3.5.14]{paternain_salo_uhlmann_2023}. Recall the ``fan-beam projection'' map 
\begin{align}   
    \pif\colon S\Dm\to \inward,    \label{eq:piF}
\end{align}
which maps a point $(z,\theta)\in S\Dm$ to the fan-beam coordinates of the unique oriented geodesic through $\Dm$ passing through $(z,\theta)$. 

We now prove Proposition \ref{prop:Ibounded}.
\begin{proof}[Proof of Proposition \ref{prop:Ibounded}] As a preliminary computation, observe that
\begin{align}
(Id^\gamma)(\beta,\alpha) = \int_0^{2\cos\alpha}\!\!\!\!d(\gamma_{\beta,\alpha}(t))^\gamma \d t \stackrel{\eqref{eq:d_along_flow}}{=}
	\int_0^{2\cos\alpha}\!\!\!\! (2\mu t-t^2)^\gamma \d t = (2\mu)^{2\gamma+1} \!\overbrace{\int_0^1 s^\gamma(1-s)^\gamma \d u}^{B(\gamma+1,\gamma+1)},
    \label{eq:Iofdgamma}
\end{align}
where in the last equality, we have changed variable $t = 2s\cos\alpha$ and we write $\mu = \cos \alpha$ for short. 

On to the continuity estimate, for $u\in C^\infty(\SD)$, we compute
\begin{align}
    \|Id^\gamma u\|^2_{L^2_\gamma(\inward)}
    &= \int_\inward |Id^\gamma u|^2 \mu^{-2\gamma}\d\beta \d\alpha\\
    &= \int_\inward \left|\int_0^{2\cos\alpha} \!\!u(\gamma_{\beta,\alpha}(t),\dot\gamma_{\beta,\alpha}(t)) d^\gamma(\gamma_{\beta,\alpha}(t))\d t \right|^2
	\!\!\! \mu^{-2\gamma}\d\beta \d\alpha\\
    &\leq \int_\inward \int_0^{2\cos\alpha}	|u|^2 d^\gamma \d t\ (Id^\gamma(\beta,\alpha))\ \mu^{-2\gamma}\d\beta \d\alpha\\
    &\stackrel{\eqref{eq:Iofdgamma}}{\le} 2^{2\gamma+1} B(\gamma{+}1,\gamma{+}1)\int_\inward \int_0^{2\cos\alpha} |u|^2 d^\gamma \d t\ 
\mu \d\beta \d\alpha\\
&\stackrel{\eqref{eq:Santalo}}{\le} 2^{2\gamma+1} B(\gamma{+}1,\gamma{+}1)
	\int_\SD \!\!|u|^2 d^\gamma \ \d\Sigma^3,
\end{align}
hence the continuity estimate, which by density extends to $L^2_\gamma(S\Dm)$. 

To prove surjectivity, we construct a bounded right-inverse for $Id^\gamma$. Let $u\in C^\infty(\inward)$ be arbitrary. With $\pif$ defined in \eqref{eq:piF}, evaluate the action of $Id^\gamma$ on the first integral
\begin{align}
I d^\gamma \left( \frac{u}{\mu^{2\gamma+1}}\circ \pif \right) =
\frac{u}{\mu^{2\gamma+1}} I d^\gamma &\stackrel{\eqref{eq:Iofdgamma}}{=} \frac{u}{\mu^{2\gamma+1}}(2\mu)^{2\gamma+1}B({\gamma+1},{\gamma+1})\\
&= 2^{2\gamma+1}B(\gamma+1,\gamma+1)u.
\end{align}
Thus, the $d^\gamma$-weighted X-ray transform sends $((u/\mu^{2\gamma+1})\circ\pif)/(2^{2\gamma{+}1}B(\gamma{+}1,\gamma{+}1))$ to $u$. That the map $u\mapsto ((u/\mu^{2\gamma+1})\circ\pif)/(2^{2\gamma{+}1}B(\gamma{+}1,\gamma{+}1))$ extend boundedly to a map $L^2_\gamma(\inward)\to L^2_\gamma(S\Dm)$ follows from the estimate
\begin{align}
\left\|\frac{u}{\mu^{2\gamma+1}}\circ\pif\right\|^2
	_{L^2_\gamma(S\Dm)}&=
\int_{S\Dm} \left( \frac{u}{\mu^{2\gamma+1}}\circ \pif \right)^2 d^\gamma
	\d\Sigma^3 \\
&\! \stackrel{\eqref{eq:Santalo}}{=} \int_\inward \frac{u^2}{\mu^{4\gamma+2}} I(d^\gamma) \mu\d\beta\d\alpha\\
&\! \stackrel{\eqref{eq:Iofdgamma}}{=} \int_\inward \frac{u^2}{\mu^{4\gamma+2}} (2\mu)^{2\gamma+1} B(\gamma{+}1,\gamma{+}1) \mu\d\beta\d\alpha \\
&= 2^{2\gamma+1}B(\gamma{+}1,\gamma{+}1)\int_\inward
	u^2\mu^{-2\gamma}\d\beta\d\alpha.
\end{align}
The proof of Proposition \ref{prop:Ibounded} is complete.
\end{proof}

\section{Tensor decompositions - Proof of Theorems \ref{thm:Yorkdecomp}, \ref{thm:rangedecomp} and Corollary \ref{cor:decomposition}} \label{sec:gauge}

\subsection{A singularly-weighted elliptic decomposition} \label{sec:elldecunitdisk}

The main result of this section is an elliptic decomposition of functions in $L^2(\Dm,d^\gamma |\d z|^2)$ for a weighted $\partial$ operator. 

\begin{proposition}\label{prop:decomp}
Fix $\gamma>-1$. Any $f\in L_\gamma^2(\mathbb{D})$ admits a unique $L^2_\gamma(\Dm)$-orthogonal decomposition
\begin{align*}
f = d^{-\gamma} \partial v + g,
\end{align*}
for some $v\in H^{1,-\gamma}_0(\mathbb{D})$ and $g\in L_\gamma^2(\mathbb{D})\cap \ker\dbar$. Moreover, one has the stability estimate
\begin{align}
\|v\|_{H_0^{1,-\gamma}(\mathbb{D})} + \|g\|_{L^2_\gamma(\mathbb{D})}
	\leq C\|f\|_{L^2_\gamma(\mathbb{D})}.
    \label{eq:stability}
\end{align}
By complex conjugation, there exists a unique decomposition $f=d^{-\gamma}\dbar v' + g$ where $\partial g'=0$ satisfying a similar stability estimate.
\end{proposition}

The proof of Proposition \ref{prop:decomp} relies on a Poincar\'e inequality proved in Appendix \ref{sec:poincare}.

\begin{lemma}[Poincar\'e inequality]\label{lem:poincare}
Let $\gamma\in(-1,1)$. Then there exists a constant $C$ such that
    \begin{align}
        \|u\|_{L^2_\gamma(\mathbb{D})} \leq C\|\nabla u\|_{L^2_\gamma(\mathbb{D})}, \qquad u\in H^{1,\gamma}_0(\mathbb{D}).
    \end{align}
\end{lemma}

\begin{proof}[Proof of Proposition \ref{prop:decomp}]
Consider the symmetric bilinear form $B(v, u) = (\partial v, -\partial u)_{L^2_{-\gamma}(\mathbb{D})}$ which is bounded on $H^{1,-\gamma}_0(\mathbb{D})$. As a consequence of Lemma \ref{lem:poincare}, this form is coercive. The linear functional $f^* \colon H^{1,-\gamma}_0(\mathbb{D}) \to \mathbb{R}$ defined by $f^*(u)=(f, -\partial u)_{L^2(\mathbb{D})}$ is also bounded by the estimate
\begin{align}
|f^*(u)|
= |(f, \partial u)_{L^2(\mathbb{D})}|
= |(fd^{\gamma/2},d^{-\gamma/2}\partial u)_{L^2(\mathbb{D})}|
\leq \|f\|_{L^2_\gamma(\mathbb{D})}\|u\|_{H^{1,-\gamma}_0(\mathbb{D})}.
\end{align}
By the Riesz Representation Theorem, there exists a unique $v\in H^{1,-\gamma}_0(\mathbb{D})$ such that
\begin{align}
f^*(u) = B(v,u) \quad \text{for all $u\in H^{1,-\gamma}_0(\mathbb{D})$},
\end{align}
together with the norm estimate $\|v\|_{H^{1,-\gamma}_0(\mathbb{D})}=\|f^*\| \leq\|f\|_{L^2_\gamma(\mathbb{D})}$. Since we have that $(f,-\partial u) = (\partial v, -\partial u)_{L^2_{-\gamma}(\Dm)}$ for every $u\in C_c^\infty(\Dm^\circ)$, then we have $\dbar (f - d^{-\gamma} \partial u) =0$ in the sense of distributions. Thus, setting $g:=f-d^{-\gamma}\partial v$, we have $\dbar g=0$, which by elliptic regularity implies $g\in C^\infty(\Dm^\circ)$. Further, since both $f$ and $d^{-\gamma} \partial v$ belong to $L^2_\gamma(\Dm)$, so does $g$, hence $g\in L^2_\gamma(\Dm)\cap \ker\dbar$.

Finally, we show the orthogonality statement
\begin{align}
    (d^{-\gamma}\partial v, g)_{L^2_\gamma(\Dm)} = 0, \quad v\in H^{1,-\gamma}_0 (\Dm), \qquad g\in L^2_\gamma(\Dm) \cap \ker \dbar,
    \label{eq:orthogonality}
\end{align}
which in particular implies the stability estimate \eqref{eq:stability}. To see this, notice that the left side is a bounded bilinear form on $H^{1,-\gamma}_0(\Dm)\times (L^2_\gamma(\Dm) \cap \ker\dbar)$, vanishing identically on the dense subspace $C_c^\infty(\Dm^\circ) \times (C^\infty(\Dm)\cap \ker\dbar)$, as can be seen from a integration by parts with no boundary term. Hence this identity holds on $H^{1,-\gamma}_0(\Dm)\times (L^2_\gamma(\Dm) \cap \ker\dbar)$.
\end{proof}

\subsection{Tensor fields and fiberwise Fourier analysis} \label{sec:fibFourier}

In order to apply Proposition \ref{prop:decomp} to produce tensor field decompositions, we first recall how to view symmetric tensors under the lens of fiberwise Fourier analysis. Through the global chart \eqref{eq:SDchart} on $S\Dm$, functions in $S\Dm$ can be represented in terms of their fiberwise Fourier decomposition 
\begin{align}
    f(z,\theta) = \sum_{k\in \Zm} f_k(z) e^{ik\theta}, \qquad f_k(z) = \frac{1}{2\pi}\int_0^{2\pi} f(z,\theta)\d\theta. 
    \label{eq:fibFourier}
\end{align}
When $f\in C^\infty(S\Dm)$ (resp. $L^2_\gamma (S\Dm)$), then for every $k\in \Zm$, $f_k\in C^\infty(\Dm)$ (resp. $L^2_\gamma (\Dm)$). This induces decompositions
\begin{align}
    C^\infty(S\Dm) &= \bigoplus_{k\in \Zm} \Omega_k, \qquad \Omega_k := \{f(z) e^{ik\theta},\ f\in C^\infty(\Dm)\}, \\
    L^2_\gamma(S\Dm) &= \bigoplus_{k\in \Zm} H_{\gamma,k}, \qquad H_{\gamma,k} := \{f(z) e^{ik\theta},\ f\in L^2_\gamma (\Dm)\}.
\end{align}
The geodesic and transverse vector fields 
\begin{align*}
    X = \cos\theta \partial_x + \sin\theta\partial_y, \qquad X_\perp = [X,\partial_\theta] = \sin\theta \partial_x - \cos\theta \partial_y,
\end{align*}
can be written in terms of the so-called Guillemin-Kazhdan operators $\eta_\pm = \frac{1}{2} (X\pm iX_\perp)$, introduced in \cite{Guillemin1980} in the context of Riemannian surfaces, and whose expression is given by
\begin{align}
\eta_-u = (\partial u)e^{-i\theta}\quad\text{and}\quad
\eta_+u = (\dbar u)e^{i\theta}.
\end{align}
In particular, we have $\eta_\pm (\Omega_k)\subset \Omega_{k\pm 1}$ for all $k\in \Zm$. 

For $m\ge 0$, it is well-known (see, e.g. \cite[Chapter 6]{paternain_salo_uhlmann_2023}) that the map $\ell_m$ defined in \eqref{eq:ellm} induces an isomorphism, 
\begin{align}
    \ell_m &\colon L^2_\gamma (\Dm, S^m(T^*\Dm)) \stackrel{\approx}{\longrightarrow} \bigoplus_{k=0}^m H_{\gamma,m-2k}, \label{eq:ellm_iso1}
\end{align}
and, in addition, for traceless tensors and tt tensors:
\begin{align}
    \ell_m &\colon L^2_\gamma (\Dm, S^m(T^*\Dm)) \cap \ker \tr \stackrel{\approx}{\longrightarrow} H_{\gamma,-m} \oplus H_{\gamma,m}, \quad (m\ge 2), \label{eq:ellm_iso2} \\
    \ell_m &\colon L^2_\gamma (\Dm, \Stt^m(T^*\Dm)) \stackrel{\approx}{\longrightarrow} (H_{\gamma,-m}\cap \ker \eta_+) \oplus (H_{\gamma,m} \cap \ker \eta_-), \quad (m\ge 1). \label{eq:ellm_iso3}
\end{align}

From \cite[Lemma 6.3.2]{paternain_salo_uhlmann_2023}, one has
\begin{align}
    \begin{split}
        X\circ \ell_m &= \ell_{m+1}\circ \d^s \quad \text{on} \quad C^\infty(\Dm; S^m(T^*\Dm)), \quad m\ge 0, \\
        X_\perp\circ\ell_0 &= -\ell_1 \circ \star\d \quad \text{on} \quad C^\infty(\Dm),
    \end{split}
    \label{eq:ellintertwiners}
\end{align}
where $\star$ is defined in \eqref{eq:Hodge}.

\subsection{Decompositions. Proofs of Lemma \ref{lem:oneformdecompSD}, Theorem \ref{thm:Yorkdecomp}, and Corollary \ref{cor:decomposition}} \label{sec:proofdecomp}

We begin with the decomposition results for one-forms.

\begin{proof}[Proof of Lemma \ref{lem:oneformdecompSD}] Let $w \in L^2_\gamma(\Dm, S^1(T^*\Dm))$, then $\ell_1 w = w_{-1} e^{-i\theta} + w_1 e^{i\theta}$. 

(1) Apply the decomposition of Proposition \ref{prop:decomp}
and its complex conjugate to obtain
\begin{align}
w_1 = d^{-\gamma}\partial v_+ + g_1 \quad\text{and}\quad
w_{-1} = d^{-\gamma} \overline{\partial} v_{-} + g_{-1},
\end{align}
where $v_\pm\in H^{1,-\gamma}_0(\Dm)$, $g_{\pm 1}\in L^2_\gamma(\Dm)$,
and $\overline{\partial} g_{1} = \partial g_{-1}=0$. We then write
\begin{align}
    \ell_1 w &= w_{-1} e^{-i\theta} + w_1 e^{i\theta} \\
    &= (d^{-\gamma} \overline{\partial} v_{-} + g_{-1}) e^{-i\theta} + (d^{-\gamma}\partial v_+ + g_1 ) e^{i\theta}\\
    &= d^{-\gamma}(\eta_- v_{-} + \eta_+ v_+) + \ell_1 \tilde{g}_1,
\end{align}
where $\tilde{g}_1 \coloneq g_1\d z + g_{-1}\d\zbar\in L^2(\Dm, \Stt^1(T^*\Dm))$. We now use that $\eta_\pm = \frac{1}{2} (X\pm i X_\perp)$ to write
\begin{align}
    \eta_- v_{-} + \eta_+v_+ &= X\left(\frac{v_+ +v_-}{2}\right)\! + X_\perp\!\left( \frac{v_+-v_-}{2i} \right) \stackrel{\eqref{eq:ellintertwiners}}{=} \ell_1 \left( \d \frac{v_++v_{-}}{2} + \star\d \frac{v_--v_{+}}{2i}\right),
\end{align}
and hence setting $g_0 \coloneq (v_+ +v_{-})/2$ and
$g_s\coloneq (v_--v_{+})/(2i)$ completes the proof. 

(2) Apply the decomposition of Proposition \ref{prop:decomp} to $w_{-1}$, writing
\[ w_{-1} = d^{-\gamma} \dbar f + \tilde{w}_{-1}, \quad f\in H^{1,-\gamma}_0(\Dm),\ \tilde{w}_{-1} \in L^2_\gamma(\Dm)\cap \ker \partial, \]
and observe that 
\begin{align*}
    \ell_1 w &= w_{-1} e^{-i\theta} + w_1 e^{i\theta} \\
    &= (d^{-\gamma} \dbar f + \tilde{w}_{-1}) e^{-i\theta} + w_1 e^{i\theta} \\
    &= d^{-\gamma} X f + \tilde{w}_{-1}e^{-i\theta} + (w_1 - \partial f) e^{i\theta} \\
    &= \ell_1 (d^{-\gamma} \d f + \tilde{w}_{-1}\d\zbar + \tilde{w}_1 \d z),
\end{align*}
upon defining $\tilde{w}_1:= w_1 - \partial f\in L^2_\gamma(\Dm)$. The proof is complete.
\end{proof}

We now move to higher-order tensor field decompositions and the proof of Theorem \ref{thm:Yorkdecomp}.

\begin{proof}[Proof of Theorem \ref{thm:Yorkdecomp}]\label{pf:Yorkproof}
Let $f\in L^2_\gamma(\Dm; S^m(T^*\Dm))$ and write 
\begin{align*}
    \ell_m f = f_{-m} e^{-im\theta} + \ell_{m} Lh + f_{m} e^{im\theta}, \qquad f_{\pm m} \in L^2_\gamma(\Dm),\quad h\in L^2_\gamma(\Dm; S^{m-2}(T^*\Dm)).
\end{align*}
By Proposition \ref{prop:decomp}, decompose $f_m,f_{-m}$ as 
\begin{align*}
    f_m = d^{-\gamma}\partial q_{m-1} + g_m, \qquad f_{-m} = d^{-\gamma}\overline{\partial} q_{-m+1} + g_{-m},
\end{align*}
where $q_{\pm(m-1)}\in H^{1,-\gamma}_0(\Dm)$ and $g_{\pm m}\in L^2_\gamma(\Dm)$ satisfy $\dbar g_m = \partial g_{-m}=0$. We write
\begin{align}
    f_{-m} e^{-im\theta} + f_m e^{im\theta} = g_{-m} e^{-im\theta} + g_m e^{im\theta} + d^{-\gamma} (e^{im\theta}\partial q_{m-1} + e^{-im\theta}\dbar q_{-m+1})
    \label{eq:temp}
\end{align}
The first two terms can be recast as a tt $m$-tensor: denoting $\tilde{f} \coloneq g_m \d z^m + g_{-m}\d\zbar^m$, we have $\tilde{f}\in L^2_\gamma (\Dm, \Stt^m(T^*\Dm))$ and $\ell_m \tilde{f} = g_m e^{im\theta} + g_{-m} e^{-im\theta}$. Using $X = e^{i\theta}\partial+e^{-i\theta}\dbar$, we work on the last two terms in \eqref{eq:temp}, writing
\begin{align*}
    e^{im\theta}\partial q_{m-1} &+ e^{-im\theta}\dbar q_{-m+1} \\
    &= X(e^{i(m-1)\theta}q_{m-1}+e^{-i(m-1)\theta}q_{-m+1}) - (e^{i(m-2)\theta} \dbar q_{m-1} + e^{i(-m+2)\theta} \partial q_{-m+1}) \\
    &\!\! \stackrel{\eqref{eq:ellintertwiners}}{=} \ell_m \d^s (q_{m-1}\d z^{m-1} + q_{-m+1}\d\zbar^{m-1}) - \ell_m L (\dbar q_{m-1} \d z^{m-2} + \partial q_{-m+1} \d\zbar^{m-2}),
\end{align*}
hence the decomposition $f = d^{-\gamma} \d^s q + L \lambda + \tilde{f}$ follows upon setting 
\[ q\coloneq q_{m-1}\d z^{m-1} + q_{-m+1}\d\zbar^{m-1} \in H^{1,-\gamma}_0 (\Dm, S^{m-1}(T^*\Dm)) \cap \ker \tr, \]
and $\lambda\coloneq h - \dbar q_{m-1} \d z^{m-2} + \partial q_{-m+1} \d\zbar^{m-2} \in L^2_\gamma(\Dm, S^{m-2}(T^*\Dm))$.
\end{proof}

We now prove Corollary \ref{cor:decomposition} by iteration of Theorem \ref{thm:Yorkdecomp}.
\begin{proof}[Proof of Corollary \ref{cor:decomposition}]
In base cases $m=1$ and $m=2$, Lemma \ref{lem:oneformdecompSD} and
Theorem \ref{thm:Yorkdecomp} yield decompositions in the desired forms, respectively. Now that base cases are satisfied, suppose tensors of rank strictly less than $m$ decompose as in the statement. For $f\in L^2_\gamma(\Dm,S^m(T^* \Dm))$ with $m\ge 2$, apply Theorem \ref{thm:Yorkdecomp} to decompose
\begin{align}
f = d^{-\gamma}\d^s q_{m-1} + L\lambda + \widetilde{f}_m, \label{eq:fyorkdecomposed}
\end{align}
where $q_{m-1}\in H^{1,-\gamma}_0(\Dm, S^{m-1}(T^* \Dm))$ satisfies $\tr(q_{m-1}) = 0$, $\widetilde{f}_m\in L^2_\gamma (\Dm, \Stt^m(T^* \Dm))$ and $\lambda\in L^2_\gamma (\Dm; S^{m-2}(T^*\Dm))$. Applying the induction hypothesis to $\lambda$, there exist $q\in H^{1,-\gamma}_0(\Dm, S^{m-3}(T^* \Dm))$ and $\lambda^{\rm itt}\in L^2_\gamma(\Dm; S^{m-2}(T^* \Dm))$ (of either form \eqref{eq:fitt_even} or \eqref{eq:fitt_odd} depending on the parity of $m$) such that 
\begin{align*}
    \lambda = d^{-\gamma} \d^s q + \lambda^{\rm itt}.
\end{align*}
Returning to $f$, and using that $L\circ (d^{-\gamma}\d^s) = (d^{-\gamma}\d^s)\circ L$, this gives the decomposition
\begin{align*}
    f = d^{-\gamma} \d^s (q_{m-1}+Lq_{m-3}) + L \lambda^{\rm itt} + \widetilde{f}_m.
\end{align*}
Setting $\fitt := L \lambda^{\rm itt} + \widetilde{f}_m$, we find that if $\lambda^{\rm itt}$ takes the form \eqref{eq:fitt_even} if $m-2$ is even (resp. \eqref{eq:fitt_odd} if $m-2$ is odd), then $\fitt$ takes a similar form at degree $m$, hence fulfilling the induction step. 
\end{proof}

\section{Range decomposition} \label{sec:rangeDecomp}

A natural family of orthogonal polynomials parameterizing $L^2_\gamma(\Dm)$ is the generalized Zernike polynomials, which are recovered by application of the backprojection operator.

\begin{definition}[Generalized Zernike Polynomials]\label{def:Znk}
With $\psi_{n,k}^\gamma$ defined in \eqref{def:psink}, for $n\geq 0$ and $0\leq k\leq n$, we define 
\begin{align}
Z_{n,k}^\gamma(z) := \int_0^{2\pi} \mu^{-2\gamma-1}\psi_{n,k}^\gamma (\pif(z,\theta))\ \d\theta, \quad z\in \Dm,
\label{eq:Znk}
\end{align}
where $\pif$ is defined in Eq. \eqref{eq:piF}. Such functions can be seen to be proportional to the generalized disk Zernike polynomials as defined in  \cite{Wuensche2005}. It was shown in \cite[Theorem 1]{Mishra2022} that the functions $\{Z_{n,k}^\gamma\}_{n\ge 0\le k\le n}$ defined Eq. \eqref{eq:Znk} form a basis of $L^2_\gamma(\Dm)$. Moreover, in the present normalization of $\psi_{n,k}^\gamma$, the Singular Value Decomposition of $I_0 d^\gamma$ reads
\begin{align}
    I_0 d^\gamma \wtZ_{n,k}^\gamma = \sigma_{n,k}^\gamma\psi_{n,k}^\gamma, \qquad n\ge 0,\ 0\le k\le n,
    \label{eq:SVDI0}
\end{align}
where $\sigma_{n,k}^\gamma>0$ is uniquely defined by \eqref{eq:I0range}. Moreover, in the present normalization convention, we have
\begin{align}
    \|Z_{n,k}^\gamma\|_{L^2_\gamma}^2 = (\sigma_{n,k}^\gamma)^2, \quad n\ge 0,\ 0\le k\le n. 
    \label{eq:normZnk}
\end{align}
 \end{definition}

\subsection{Description of the range over tt $m$-tensors}\label{ssec:rangeDecomptt}
To deduce the action of $Id^\gamma$ on tensors, we first document a shifting property of the $d^\gamma$-weighted X-ray transform
\begin{lemma}\label{lem:xrayshift}
If $f\in L^2_\gamma(\Dm)$ and $p\in \Zm$, then
\begin{align}
    I d^\gamma (f e^{pi\theta}) = e^{pi(\beta+\alpha+\pi)} I_0(d^\gamma f).
    \label{eq:shifting}    
\end{align}

\end{lemma}
\begin{proof} Since the coordinate $\theta$ is a geodesically invariant quantity, this is a direct consequence of the relation
\[ \theta = (\beta+\alpha+\pi) \circ \pi_F, \]
which allows to take $e^{ip\theta}$ outside the integral defining $I d^\gamma$.
\end{proof}

Lemma \ref{lem:xrayshift} can be used to deduce the action of $I_m d^\gamma$ on a spanning set of symmetric $m$-tensors.
\begin{lemma}\label{lem:ISVD}
Suppose $n\geq 0$, $0\leq k\leq n$, and $m>0$ are integers. If $m = 2p$, then $I_m d^\gamma$ maps
\begin{align}
\wtZ_{n,k}^\gamma \d z^m &\mapsto \sigma_{n,k}^\gamma \psi_{n,k-p}^{\gamma,+}
\quad\text{and}\quad
\wtZ_{n,k}^\gamma \d \zbar^m \mapsto \sigma_{n,k}^\gamma \psi_{n,k+p}^{\gamma,+}.
\end{align}
If $m=2p+1$, then $I_m d^\gamma$ maps
\begin{align}
\wtZ_{n,k}^\gamma \d z^m &\mapsto \sigma_{n,k}^\gamma \psi_{n,k-p}^{\gamma,-}
\quad\text{and}\quad
\wtZ_{n,k}^\gamma \d \zbar^m \mapsto \sigma_{n,k}^\gamma \psi_{n,k+p+1}^{\gamma,-}.
\end{align}
\end{lemma}
\begin{proof}
Identify the basis $\{Z_{n,k}^\gamma \d z^m, Z_{n,k}^\gamma \d \zbar^m\}$ of $L^2_\gamma(\Dm; S^m(T^*\Dm))$. Let $m=2p$ be an even integer. By Eq. \eqref{eq:Stt}, one has $\ell_1(\d z) = e^{i\theta}$ and $\ell_1(\d\zbar)=e^{-i\theta}$. Then
\begin{align}
I_m d^\gamma \wtZ_{n,k}^\gamma \d z^m
&= I(d^\gamma \ell_1(\wtZ_{n,k}^\gamma\d z^m))
= I(d^\gamma \wtZ_{n,k}^\gamma e^{mi\theta})\\
&= I(d^\gamma \wtZ_{n,k}^\gamma)e^{mi(\beta+\alpha+\pi)}
= \sigma_{n,k}^\gamma \psi_{n,k}^\gamma e^{2pi(\beta+\alpha+\pi)}
= \sigma_{n,k}^\gamma \psi_{n,k-p}^\gamma.
\end{align}
and
\begin{align}
I_m d^\gamma \wtZ_{n,k}^\gamma \d \zbar^m
&= I(d^\gamma \ell_1(\wtZ_{n,k}^\gamma\d \zbar^m))
= I(d^\gamma \wtZ_{n,k}^\gamma e^{-mi\theta})\\
&= I(d^\gamma \wtZ_{n,k}^\gamma)e^{-mi(\beta+\alpha+\pi)}
= \sigma_{n,k}^\gamma \psi_{n,k}^\gamma e^{-2pi(\beta+\alpha+\pi)}
= \sigma_{n,k}^\gamma \psi_{n,k+p}^\gamma.
\end{align}

If $m=2p+1$ is odd, then
\begin{align}
I_m d^\gamma \wtZ_{n,k}^\gamma \d z^m
&= I(d^\gamma \ell_1(\wtZ_{n,k}^\gamma\d z^m))
= I(d^\gamma \wtZ_{n,k}^\gamma e^{mi\theta})\\
&= I(d^\gamma \wtZ_{n,k}^\gamma)e^{mi(\beta+\alpha+\pi)}
= \sigma_{n,k}^\gamma \psi_{n,k}^\gamma e^{(2p+1)i(\beta+\alpha+\pi)}
= \sigma_{n,k}^\gamma \psi_{n,k-p}^\gamma e^{i(\beta+\alpha+\pi)}
\end{align}
and
\begin{align}
I_m d^\gamma \wtZ_{n,k}^\gamma \d \zbar^m
&= I(d^\gamma \ell_1(\wtZ_{n,k}^\gamma\d \zbar^m))
= I(d^\gamma \wtZ_{n,k}^\gamma e^{-mi\theta})
= I(d^\gamma \wtZ_{n,k}^\gamma)e^{-mi(\beta+\alpha+\pi)}\\
&= \sigma_{n,k}^\gamma \psi_{n,k}^\gamma e^{-(2p+1)i(\beta+\alpha+\pi)}
= \sigma_{n,k}^\gamma \psi_{n,k}^\gamma e^{-(2p+2-1)i(\beta+\alpha+\pi)}
= \sigma_{n,k}^\gamma \psi_{n,k+p+1}^\gamma e^{i(\beta+\alpha+\pi)}.
\end{align}
Adopting the notation of Eq. \eqref{eq:psinkpm} yields the result.
\end{proof}

\begin{lemma}\label{lem:charholomorphic}
The spaces of $L^2_\gamma(\Dm)$-integrable analytic/antianalytic functions admit the characterizations
\begin{align}
L^2_\gamma(\Dm)\cap\ker\dbar &= \left\{
\sum_{n=0}^\infty f_n z^n \;|\; \sum_{n=0}^\infty |f_n|^2B(n{+}1,\gamma{+}1) < \infty
\right\} \quad\text{and}\\
L^2_\gamma(\Dm)\cap\ker \partial &= \left\{
\sum_{n=0}^\infty f_n \zbar^n \;|\; \sum_{n=0}^\infty |f_n|^2B(n{+}1,\gamma{+}1) < \infty
\right\}.
\end{align}
\end{lemma}

\begin{proof}
The case $\gamma = 0$ is classical \cite[Sec. 14.1]{Forster}. If $\dbar f = 0$ on $\Dm^\circ$, then $f$ can be expressed as a power series $\sum_{n=0}^\infty f_n z^n$ with radius of convergence at least $1$, converging uniformly on compact subsets of $\Dm^\circ$. If $f$ is further assumed $L^2_\gamma(\Dm)$-integrable, then the monotone convergence theorem applies, i.e.
\[ \|f\|^2_{L^2_\gamma(\Dm)} = \lim_{r\to 1} \int_{|z|\le r} |f|^2 d^\gamma |\d z|^2. \]
For fixed $r<1$, the power series converges uniformly on $\{|z|\le 1\}$, and by direct calculation  
\begin{align*}
    \int_{|z|<r} |f|^2 d^\gamma |\d z|^2 = \sum_{n=0}^\infty |f_n|^2 \int_{|z|\le r} |z|^{2n} d^\gamma(z) |dz|^2 &\stackrel{(z=\rho e^{i\omega})}{=} 2\pi \sum_{n=0}^\infty |f_n|^2 \int_0^r \rho^{2n} (1-\rho^2) \ \rho\d\rho \\
    &\ \ \stackrel{(s=\rho^2)}{=} \pi \sum_{n=0}^\infty |f_n|^2 \int_0^{r^2} s^{n} (1-s)^\gamma \ \d\rho,
\end{align*}
and the result follows upon sending $r\to 1$ and using that $\int_0^1 s^{n} (1-s)^\gamma \ \d\rho \stackrel{\eqref{eq:beta}}{=} B(n+1,\gamma+1)$.

The characterization of $L^2_\gamma(\Dm)\cap \ker \partial$ is then obtained by complex-conjugation. 
\end{proof}

Combining the previous with Lemma \ref{lem:ISVD} allows one to conclude the following isomorphisms.
\begin{lemma}\label{lem:isomIkerk}
    Let $p\ge 0$. The restrictions
    \begin{align}
        I d^\gamma &\colon (H_{\gamma,2p} \cap \ker \eta_-) \to h^{(1+\gamma)/2} ( \psi_{n,-p}^{\gamma,+}, n\ge 0), \quad\text{and}\\
        I d^\gamma &\colon (H_{\gamma,2p+1} \cap \ker \eta_-) \to h^{(1+\gamma)/2} ( \psi_{n,-p}^{\gamma,-}, n\ge 0),
    \end{align}
    are isomorphisms. 
\end{lemma}

\begin{proof}
    An element $f\in H_{\gamma,2p} \cap \ker\eta_-$ can be written as $f = e^{2ip\theta} \sum_{n=0}^\infty f_n \wtZ_{n,0}^\gamma$ with $\sum_{n\ge 0} |f_n|^2<\infty$. Then by Lemma \ref{lem:xrayshift},
    \begin{align}
        I d^\gamma f = e^{2ip(\beta+\alpha+\pi)} \sum_{n=0}^\infty f_n I_0 d^\gamma \wtZ_{n,0}^\gamma= \sum_{n=0}^\infty \sigma_{n,0}^\gamma f_n\psi_{n,-p}^{\gamma,+}.
    \end{align}
    To verify the right-hand side lies in $h^{(1+\gamma)/2} ( \psi_{n,-p}^{\gamma,+}, n\ge 0)$, make the estimate
    \begin{align}
    \sum_{n=0}^\infty |\sigma_{n,0}^\gamma f_n|^2 (n+1)^{\gamma+1}
    \sim \sum_{n=0}^\infty |f_n|^2 < \infty
    \end{align}
    by applying the asymptotic $(\sigma_{n,0}^\gamma)^2 \sim n^{-\gamma-1}$ as $n\to\infty$ (see, e.g., \cite[Sec. 5.2]{Mishra2022}). Note in addition that $f$ can be recovered by the formula 
    \begin{align}
        f = e^{2ip\theta}\sum_{n=0}^\infty \frac{(I d^\gamma f, \psi_{n,-p}^{\gamma,+})}{\sigma_{n,0}^\gamma} \wtZ_{n,0}^\gamma.
        \label{eq:frecons}
    \end{align}
    
    The odd case is similar, using that $e^{(2p+1)i(\beta+\alpha+\pi)} \psi_{n,0}^\gamma = \psi_{n,-p}^{\gamma,-}$. 
\end{proof}

\begin{proof}[Proof of Lemma \ref{lem:homeomorphisms}]
Recall that for $k\ge 1$, Equation \eqref{eq:ellm_iso3} reads
\begin{align*}
    L^2_\gamma (\Dm, \Stt^k (T^*\Dm)) \cong (H_{\gamma,k}\cap \ker \eta_-) \oplus (H_{\gamma,-k}\cap \ker\eta_+) = (H_{\gamma,k}\cap \ker \eta_-) \oplus \overline{(H_{\gamma,k}\cap \ker \eta_-)}.
\end{align*}
Hence Lemma \ref{lem:homeomorphisms} follows by linearity, a direct application of Lemma \ref{lem:isomIkerk}, and the fact that for all $n\ge 0$ and $k\in \Zm$, $\overline{\psi_{n,k}^{\gamma,+}} = \psi_{n,n-k}^{\gamma,+}$ and $\overline{\psi_{n,k}^{\gamma,-}} = \psi_{n,n-k+1}^{\gamma,-}$. For $k=2p$ even, this implies 
\begin{align*}
    L^2_\gamma (\Dm, \Stt^k (T^*\Dm)) \subset h^{(1+\gamma)/2} ( \psi_{n,-p}^{\gamma,+}, n\ge 0) + h^{(1+\gamma)/2} ( \psi_{n,n+p}^{\gamma,+}, n\ge 0).
\end{align*}
The fact that the sum is orthogonal on the right side follows from the fact that the index spans are disjoint. When $k$ is odd, a similar argument applies.
\end{proof}

\subsection{Decomposition of the range of $I_1 d^\gamma$}

The first integrals decomposed in the next lemma are used to prove Corollary \ref{cor:decomposition} and in the proof of Theorem \ref{thm:rangedecomp}.
\begin{lemma}\label{lem:psinkpreimage}
For $\gamma >-1$, $n\ge 0$ and $k\in \Zm$, one has the sums
\begin{align}
    (\mu^{-2\gamma-1}\psi_{n,k}^{\gamma,+})\circ\pif(z,\theta) &= \sum_{j=0}^n e^{2(j-k)i\theta}Z_{n,j}^\gamma(z), \qquad \text{and}
    \label{eq:psinkpreimage}\\
    (\mu^{-2\gamma-1}\psi_{n,k}^{\gamma,-})\circ\pif(z,\theta) &= \sum_{j=0}^n e^{(2j-2k+1)i\theta}Z_{n,j}^\gamma(z), \qquad \text{for $(z,\theta)$} \in S\Dm.
    \label{eq:psinkpreimageodd}
\end{align}
Moreover, if $\hat L_n^\gamma(x) = \hat \ell_n^\gamma x^n + \cdots$ and $Z_{n,k}^\gamma = \hat g_{n,k}^\gamma z^{n-k}\zbar^k + \cdots$, then $\hat g_{n,k}^\gamma = \hat \ell_n^\gamma\binom{n}{k}(-1)^k/(2i)^n$.
\end{lemma}

\begin{proof}
Start with expression \eqref{def:psink} of the fan-beam polynomials
\begin{align}
\psi_{n,k}^\gamma(\beta,\alpha) \coloneq \mu^{2\gamma+1}
	e^{i(n-2k)(\beta+\alpha+\pi)}\hat L_n^\gamma(\sin\alpha)/2\pi.
\end{align}
Denoting $\pi_F(z,\theta) = (\beta_-,\alpha_-)\in \partial_+ S\Dm$, apply the relations
\[ \beta_- + \alpha_- + \pi = \theta \quad \text{and} \quad \sin\alpha_-(z,\theta) = \frac{z e^{-i\theta} - \zbar e^{i\theta}}{2i},  \]
to obtain
\begin{align}
(\mu^{-2\gamma-1}\psi_{n,k}^\gamma)\circ\pif(\rho e^{i\omega},\theta)
= e^{(n-2k)i\theta} \hat L_n^\gamma ((z e^{-i\theta} - \zbar e^{i\theta})/(2i))/2\pi.
\end{align}
Hence \eqref{eq:psinkpreimage} will be proved if we can show that 
\begin{align}
    \hat L_n^\gamma(\sin\alpha_-(z,\theta)) = \sum_{j=0}^n e^{-(n-2j)i\theta}Z_{n,j}^\gamma(z), 
    \label{eq:jacobisinalpha-}
\end{align}
which we now prove. That the Fourier support of $\hat L_n^\gamma(\sin\alpha_-(z,\theta))$ is no more than $\{ e^{i(n-2j)\theta},\ 0\le j\le n\}$ is a direct consequence of the fact that $\hat L_n^\gamma(\sin\alpha_-(z,\theta))=\hat L_n^\gamma((ze^{-i\theta}{-}\zbar e^{i\theta})/(2i))$ is a polynomial of degree $n$ with the same parity as $n$ of a homogeneous polynomial of degree $1$ in $e^{\pm i\theta}$. The expression of the nonzero Fourier modes then follows directly from Equation \eqref{eq:Znk}, namely that, for $n\ge 0$ and $0\le j\le n$,
\begin{align}
Z_{n,j}^\gamma(z) = \frac{1}{2\pi}\int_{\mathbb{S}^1}\mu^{-2\gamma-1}\psi_{n,j}^\gamma \circ\pi_F(z,\theta)\d \theta = \frac{1}{2\pi}\int_{\mathbb{S}^1}
    e^{(n-2j)i\theta}\hat L_n^\gamma(\sin\alpha_-(z,\theta))\d\theta.
\end{align}

Finally, to deduce a relation between the constants $\hat \ell_n^\gamma$ and $\hat g_n^\gamma$, decompose the polynomials as in the statement. Apply the binomial theorem to obtain
\begin{align}
\hat L_n^\gamma(\sin\alpha_-)
&= \hat\ell_n^\gamma \left(\frac{ze^{-i\theta}-\zbar e^{i\theta}}{2i}\right)^n+\cdots \\
&= \frac{\hat\ell_n^\gamma}{(2i)^n} \sum_{j=0}^n \binom{n}{j}\zbar^j e^{ji\theta} z^{n-j}e^{-(n-j)i\theta}(-1)^j+\cdots\\
&= \frac{\hat\ell_n^\gamma}{(2i)^n} \sum_{j=0}^n \binom{n}{j}\zbar^j z^{n-j}e^{-(n-2j)i\theta}(-1)^j+\cdots.
\end{align}
Identify the previous coefficients with coefficients in Equation \eqref{eq:jacobisinalpha-} to obtain the result.
\end{proof}

\begin{lemma}\label{lem:rangeoneforms}
The range of the operator $I_1d^\gamma \colon L^2_{\gamma}(\Dm; S^1(T^*\Dm))
\to L^2_{\gamma,-}(\inward)$ decomposes orthogonally
\begin{align}
I_1d^\gamma(L^2_\gamma(\Dm; S^1(T^*\Dm))) = I_1({\star}\d(H^{1,-\gamma}_0(\Dm))) \stackrel{\perp}{\oplus} I_1 d^\gamma(L^2_\gamma(\Dm, \Stt^1(T^*\Dm))).
\label{eq:orthodecomp}
\end{align}
\end{lemma}

\begin{proof}
Let $u=I_1(d^\gamma w)$ for some $w\in L^2_\gamma(\Dm; S^1(T^* \Dm))$.
By Lemma \ref{lem:oneformdecompSD}, decompose
\begin{align}
w = d^{-\gamma} \d^s g_0 + d^{-\gamma} {\star}\d g_s + \tilde{g}_1,
\end{align}
where $g_0,g_s\in H^{1,-\gamma}_0(\Dm)$ and $\tilde g_{1}\in L^2_\gamma(\Dm, \Stt^1(T^*\Dm))$.
Then
\begin{align} 
    I_1(d^\gamma w) = \cancel{I_1( \d g_0)} + I_1(\star\d g_s) + I_1 d^\gamma \tilde{g}_1,
\end{align}
where $I_1(\d g_0) = 0$ since $g_0$ vanishes on the boundary. It remains to argue that $I_1(\star\d g_s)$ and $I_1 d^\gamma \tilde{g}_1$ are orthogonal. Let us denote $v :=(\mu^{-2\gamma-1}\ I_1d^\gamma \tilde{g}_1)\circ\pif$ for short, and notice, using Santal\'o's formula that 
\begin{align}
    \int_\inward  I(X_\perp g_s) \overline{I(d^\gamma g_1)}	\mu^{-2\gamma} \d\Sigma^2 = \int_\inward  I(X_\perp g_s \overline{v}) \mu \d\Sigma^2 \stackrel{\eqref{eq:Santalo}}{=} \int_\SD X_\perp g_s \overline{v}\ \d\Sigma^3
\end{align}
and using fiberwise Fourier decomposition, the last term equals
\begin{align}
    \int_\SD X_\perp g_s \overline{v}\ \d\Sigma^3 = \frac{1}{i} \left( (d^{-\gamma}\partial g_s, v_1)_{L^2_\gamma(\Dm)} - (d^{-\gamma}\dbar g_s, v_{-1})_{L^2_\gamma(\Dm)} \right).
\end{align}
In light of \eqref{eq:orthogonality}, the last right-hand side will be zero if we can show that $v_1$ and $\overline{v_{-1}}$ belong to $L^2_\gamma(\Dm)\cap\ker\partial$ and $L^2_\gamma(\Dm) \cap \ker\dbar$, respectively. To see this, recall from Lemma \ref{lem:homeomorphisms} that 
\begin{align}
    I_1 d^\gamma \tilde{g}_1 = \sum_{n=0}^\infty (a_{1,n} \psi_{n,0}^{\gamma,-} + a_{-1,n} \psi_{n,n+1}^{\gamma,-}),  \qquad \sum_{n=0}^\infty (n+1)^{\gamma+1}|a_{\pm 1,n}|^2 <\infty. \label{eq:g1image}
\end{align}
Then, 
\begin{align}
    v_{\pm 1} = \sum_{n=0}^\infty \left(a_{1,n}  (\mu^{-2\gamma-1}\psi_{n,0}^{\gamma,-}\circ \pi_F)_{\pm 1} + a_{-1,n} (\mu^{-2\gamma-1}\psi_{n,n+1}^{\gamma,-}\circ \pi_F)_{\pm 1}\right).
\end{align}
Lemma \ref{lem:psinkpreimage} is applied so that
\begin{align*}
    v_1 = \sum_{n=0}^\infty a_{1,n} Z_{n,0}^\gamma \quad\text{and}\quad v_{-1} = \sum_{n=0}^\infty a_{-1,n} Z_{n,n}^\gamma.
\end{align*}
so that 
\begin{align*}
    \|v_1\|_{L^2_\gamma}^2 &= \sum_{n=0}^\infty |a_{1,n}|^2 \|Z_{n,0}^\gamma\|^2 \stackrel{\eqref{eq:normZnk}}{=} \sum_{n=0}^\infty |a_{1,n}|^2 (\sigma_{n,0}^\gamma)^2, \\
    \|v_{-1}\|_{L^2_\gamma}^2 &= \sum_{n=0}^\infty |a_{-1,n}|^2 \|Z_{n,n}^\gamma\|^2 \stackrel{\eqref{eq:normZnk}}{=} \sum_{n=0}^\infty |a_{-1,n}|^2 (\sigma_{n,n}^\gamma)^2,
\end{align*}
where $(\sigma_{n,0}^\gamma)^2 = (\sigma_{n,n}^\gamma)^2\sim n^{-\gamma-1}$. Combining this with estimate \eqref{eq:g1image}, we find that $v_{\pm 1}\in L^2_\gamma(\Dm)$, hence \eqref{eq:orthodecomp} is proved.
\end{proof}

\begin{corollary}
    We have the range characterization
    \begin{align}
        I_1({\star}\d(H^{1,-\gamma}_0(\Dm))) = \left\{ \sum_{n=1}^\infty \sum_{k=1}^{n} c_{n,k} \psi_{n,k}^{\gamma,-},\quad \sum_{n\ge 1} \sum_{k=1}^{n} \frac{|c_{n,k}|^2}{(\sigma_{n,k}^\gamma)^2} <\infty   \right\}.        
    \label{eq:rangeIperp}    
\end{align}
\end{corollary}

\begin{proof}
    We begin by proving that 
    \begin{align}
        I_1d^\gamma(L^2_\gamma(\Dm; S^1(T^*\Dm))) = \left\{ \sum_{n=0}^\infty \sum_{k=0}^{n+1} c_{n,k} \psi_{n,k}^{\gamma,-},\ \sum_{n\ge 0} \left(\sum_{k=0}^{n} \frac{|c_{n,k}|^2}{(\sigma_{n,k}^\gamma)^2} + \frac{|c_{n,n+1}|^2}{(\sigma_{n,k}^\gamma)^2} \right) <\infty   \right\}.
        \label{eq:RanI1dgamma}
    \end{align}
    Then \eqref{eq:rangeIperp} follows from a consequence of Lemma \ref{lem:rangeoneforms}, which implies
    \begin{align*}
        I_1({\star}\d(H^{1,-\gamma}_0(\Dm))) = I_1d^\gamma(L^2_\gamma(\Dm; 
    ^1(T^*\Dm))) \cap (I_1 d^\gamma(L^2_\gamma(\Dm, \Stt^1(T^*\Dm))))^\perp,
    \end{align*}
    and the characterization of $I_1 d^\gamma(L^2_\gamma(\Dm, \Stt^1(T^*\Dm)))$ given in Lemma \ref{lem:homeomorphisms}, which implies
    \[ \overline{I_1 d^\gamma(L^2_\gamma(\Dm, \Stt^1(T^*\Dm)))}^{L^2_\gamma} = \text{span}_{L_\gamma^2} ( \psi_{n,0}^{\gamma,-}, \psi_{n,n+1}^{\gamma,-},\ n \ge 0). \]

    On to the proof of \eqref{eq:RanI1dgamma}, first observe that using the decomposition \eqref{eq:oneform2}, it suffices to consider one-forms of the form $w = \tilde{w}_{-1}\d\zbar + \tilde{w_1} \d z$, where $\tilde{w}_{\pm 1}\in L^2_\gamma(\Dm)$ and $\partial \tilde{w}_{-1} = 0$. We can thus parameterize them as 
    \[ \tilde{w}_{-1} = \sum_{n\ge 0} a_n \wtZ_{n,n}^\gamma, \qquad \tilde{w}_1 = \sum_{n\ge 0} \sum_{k=0}^n b_{n,k} \wtZ_{n,k}^\gamma,\]
    where the coefficients satisfy $\sum_{n\ge 0}(|a_n|^2 + \sum_{k=0}^n     |b_{n,k}|^2) <\infty$. For such an element, we have 
    \begin{align*}
        I_1 d^\gamma (\tilde{w_1}\d\zbar + \tilde{w}_{-1} \d z) &= e^{-i(\beta+\alpha+\pi)} I_0 d^\gamma \tilde{w}_{-1} + e^{i(\beta+\alpha+\pi)} I_0 d^\gamma \tilde{w}_1) \\
        &= \sum_{n=0}^\infty \left( e^{-i(\beta+\alpha+\pi)}  a_n \sigma_{n,n}^\gamma \psi_{n,n}^\gamma + e^{i(\beta+\alpha+\pi)} \sum_{k=0}^n b_{n,k} \sigma_{n,k}^\gamma \psi_{n,k}^\gamma \right) \\
        &= \sum_{n=0}^\infty \left( a_n \sigma_{n,n}^\gamma \psi_{n,n+1}^{\gamma,-} + \sum_{k=0}^n b_{n,k} \sigma_{n,k}^\gamma \psi_{n,k}^{\gamma,-} \right).
    \end{align*}
    Upon setting $c_{n,k} := \sigma_{n,k}^\gamma b_{n,k}$ for $0\le k\le n$ and $c_{n,n+1} := \sigma_{n,n}^\gamma a_n$, the claim follows.     
\end{proof}

\subsection{Proof of Theorem \ref{thm:rangedecomp}}

\begin{proof}[Proof of Theorem \ref{thm:rangedecomp}]
Let $k>0$. Now we show $I_{2k}(d^\gamma g_{2k})\perp I_{2k+2}(d^\gamma g_{2k+2})$.
Decompose
\begin{align}
g_{2k} &= g_{2k,-}\d \zbar^{2k} + g_{2k,+}\d z^{2k} \quad\text{and}\quad g_{2k+2} = g_{2k+2,-}\d \zbar^{2k+2} + g_{2k+2,+}\d z^{2k+2},
\end{align}
where
\begin{align}\begin{aligned}
g_{2k,-} = \sum_{n=0}^\infty g_{2k,n,-}\wtZ_{n,n}^\gamma,\quad\quad\quad\quad
&g_{2k,+} = \sum_{n=0}^\infty g_{2k,n,+}\wtZ_{n,0}^\gamma\\
g_{2k+2,-} = \sum_{n=0}^\infty g_{2k+2,n,-}\wtZ_{n,n}^\gamma, \quad\text{and}\quad
&g_{2k+2,+} = \sum_{n=0}^\infty g_{2k+2,n,+}\wtZ_{n,0}^\gamma.
\end{aligned}\end{align}
In the $+$ case, evaluate the forward operator so that
\begin{align}
I(d^\gamma g_{2k,+}) &= \sum_{n=0}^\infty
	\sigma_{n,0}^{\gamma,+} g_{2k,n,+} \psi_{n,-k}^\gamma \quad\text{and}\quad
I(d^\gamma g_{2k+2,+}) = \sum_{n=0}^\infty
	\sigma_{n,0}^{\gamma,+} g_{2k+2,n,+} \psi_{n,-k-1}^\gamma.
\end{align}
By considering the indices of the polynomials, it follows that $I(d^\gamma g_{2k,+})$ and $I(d^\gamma g_{2k+2,+})$ are orthogonal.
Equation \eqref{eq:I1span} is deduced from \eqref{eq:rangeIperp} by taking $L^2_\gamma$-closure. 
\end{proof}


\section{Range characterization - Proof of Theorem \ref{thm:rangechar}} \label{sec:rangeCharac}

We first cover the even order case. 

($\implies$) Suppose $u = I_{2p} d^\gamma f$ for some $f\in L^2_\gamma(\Dm, S^{2p}(T^*\Dm))$. Without loss of generality, we may assume $f$ in iterated-tt form \eqref{eq:fitt_even}. Then 
\[ I_{2p} d^\gamma f = \sum_{j=0}^p I_{2j}d^\gamma \tilde{f}_{2j}, \quad \tilde{f}_0\in L^2_\gamma(\Dm), \quad \tilde{f}_{2j}\in L^2_\gamma(\Dm; \Stt^{2j} (T^*\Dm)),\quad 1\le j\le p.  \]
Then from \eqref{eq:I0range} and Lemma \ref{lem:homeomorphisms}, we find that $I_{2p} d^\gamma f \in \text{span}_{L^2_\gamma} (\psi_{n,k}^{\gamma,+},\ n\ge 0,\ -p\le k\le n+p)$, hence condition (a) is satisfied. Condition (b) follows from the fact that for $1\le j\le p$, $\Pi_{2j}u = I_{2j} d^\gamma \tilde{f}_{2j}$ and Lemma \ref{lem:homeomorphisms}. Condition (c) follows directly from \eqref{eq:I0range}.

($\impliedby$) Suppose $u\in L^2_\gamma(\partial_+ S\Dm)$ satisfies (a)-(b)-(c). We write 
\[ u = \sum_{j\ge 0} \Pi_{2j} u \stackrel{(a)}{=} \sum_{j=0}^p \Pi_{2j} u.  \]
Combining condition (b) and Lemma \ref{lem:homeomorphisms}, for every $1\le j\le p$, there exists $\tilde{f}_{2j} \in L^2_\gamma (\Dm; \Stt^{2j}(T^* \Dm))$ such that $\Pi_{2j}u = I_{2j} d^\gamma \tilde{f}_{2j}$. Then combining condition (c) and \eqref{eq:I0range}, $\Pi_0 u = I_0 d^\gamma \tilde{f}_0$ for some $\tilde{f}_0\in L^2_\gamma(\Dm)$. Piecing this together, we arrive at
\begin{align*}
    u = \sum_{j=0}^p \Pi_{2j}u = \sum_{j=0}^p I_{2j} d^\gamma \tilde{f}_{2j} = I_{2p} d^\gamma \left( \sum_{j=0}^p L^{p-j} \tilde{f}_{2j}\right), 
\end{align*}
setting $f:= \sum_{j=0}^p L^{p-j} \tilde{f}_{2j}$ completes the proof. 

The case of odd tensors is completely similar, where the characterization \eqref{eq:I0range} of the range of $I_0 d^\gamma$ is replaced by the characterization \eqref{eq:rangeIperp} of $I_1 (\star\d (H^{1,-\gamma}_0(\Dm)))$, and the odd-order homeomorphisms in Lemma \ref{lem:homeomorphisms} are used instead of the even-order ones.

\section{Reconstructions} \label{sec:reconstruction}

This section is concerned with the recovery of a tensor field $f\in L^2_\gamma(\Dm; S^m(T^*\Dm))$ in iterated-tt form from ${\mathcal{D}} = I_m d^\gamma f$. 

\subsection{Reconstruction of tt components - proof of Theorem \ref{thm:reconstruction}} ${}$
\paragraph{Case $m = 2p$ for $p\ge 0$.} Following \eqref{eq:fitt_even} we write $f = \sum_{j=0}^p L^{p-j} \tilde{f}_{2j}$, and further decompose, for $1\le j\le p$, 
\[ \tilde{f}_{2j} = \sum_{n=0}^\infty (f_{2j,+,n} \wtZ_{n,0}^\gamma \d z^{2j} + f_{2j,-,n} \wtZ_{n,n}^\gamma \d \zbar^{2j}), \quad \sum_{n=0}^\infty (|f_{2j,+,n}|^2 + |f_{2j,-,n}|^2)<\infty.   \]
Then recalling $\sigma_{n,0}^\gamma=\sigma_{n,n}^\gamma$ and by direct application of Lemma \ref{lem:ISVD}, we obtain
\begin{align}
    I_{2p} d^\gamma f = I_0 d^\gamma \tilde{f}_0 + \sum_{j=1}^n \sum_{n=0}^\infty  \sigma_{n,0}^\gamma (f_{2j,+,n} \psi_{n,-j}^{\gamma,+} + f_{2j,-,n}\psi_{n,n+j}^{\gamma,+}).
\end{align}
Since all summands are orthogonal by Theorem \ref{thm:rangedecomp}, and the $\psi_{n,k}^{\gamma,+}$ basis is orthonormal, we find that 
\[ f_{2j,+,n} = \frac{(\mathcal{D},\psi_{n,-j}^{\gamma,+})}{\sigma_{n,0}^\gamma}, \quad f_{2j,-,n} = \frac{(\mathcal{D},\psi_{n,n+j}^{\gamma,+})}{\sigma_{n,0}^\gamma}, \quad 1\le j\le p,\quad n\ge 0, \]
and hence for $1\le j\le p$, the tt component $\tilde{f}_{2j}$ is recovered via the formula
\begin{align}
    \tilde{f}_{2j} := \sum_{n\ge 0} \frac{1}{\sigma_{n,0}^\gamma} ((\mathcal{D},\psi_{n,-j}^{\gamma,+}) \wtZ_{n,0}^\gamma\ \d z^{2j} + (\mathcal{D},\psi_{n,n+j}^{\gamma,+}) \wtZ_{n,n}^\gamma\ \d \zbar^{2j}).
    \label{eq:f2p_recons}
\end{align}

\paragraph{Case $m= 2p+1$ for $p\ge 0$.} Following \eqref{eq:fitt_odd}, we write $f = d^{-\gamma}{\star} \d h + \sum_{j=0}^p L^{p-j} \tilde{f}_{2j+1}$, and further decompose, for $0\le j\le p$, 
\[ \tilde{f}_{2j+1} = \sum_{n=0}^\infty (f_{2j+1,+,n} \wtZ_{n,0}^\gamma \d z^{2j+1} + f_{2j+1,-,n} \wtZ_{n,n}^\gamma \d \zbar^{2j+1}), \quad \sum_{n=0}^\infty (|f_{2j+1,+,n}|^2 + |f_{2j+1,-,n}|^2)<\infty.   \]
Then by direct application of Lemma \ref{lem:ISVD}, we obtain
\begin{align}
    I_{2p+1} d^\gamma f = I_1(\star\d v) + \sum_{j=0}^p \sum_{n=0}^\infty  \sigma_{n,0}^\gamma (f_{2j+1,+,n} \psi_{n,-j}^{\gamma,-} + f_{2j+1,-,n}\psi_{n,n+j+1}^{\gamma,-}).
\end{align}
Since all summands are orthogonal by Theorem \ref{thm:rangedecomp}, and the $\psi_{n,k}^{\gamma,-}$ basis is orthonormal, we find that 
\[ f_{2j+1,+,n} = \frac{(\mathcal{D},\psi_{n,-j}^{\gamma,-})}{\sigma_{n,0}^\gamma}, \quad f_{2j+1,-,n} = \frac{(\mathcal{D},\psi_{n,n+j+1}^{\gamma,-})}{\sigma_{n,0}^\gamma}, \quad 1\le j\le p,\quad n\ge 0, \]
and hence for $0\le j\le p$, the tt component $\tilde{f}_{2j+1}$ is recovered via the formula
\begin{align}
    \tilde{f}_{2j+1} := \sum_{n\ge 0} \frac{1}{\sigma_{n,0}^\gamma} ((\mathcal{D},\psi_{n,-j}^{\gamma,-}) \wtZ_{n,0}^\gamma\ \d z^{2j+1} + (\mathcal{D},\psi_{n,n+j+1}^{\gamma,-}) \wtZ_{n,n}^\gamma\ \d \zbar^{2j+1}).
    \label{eq:f2p1_recons}
\end{align}

We now show that these inversions can also be done through the integral formulas stated in Theorem \ref{thm:reconstruction}.

\begin{proof}[Proof of Theorem \ref{thm:reconstruction}]\label{pf:kernel}
Let $\mathcal{D}=I_md^\gamma f$ be the image of a symmetric $m$-tensor in iterated-tt form. In the even case, suppose $m=2p$ and let $j$ satisfy $1\leq j\leq p$. Recalling that $\wtZ_{n,k}^\gamma = Z_{n,k}^\gamma/\sigma_{n,k}^\gamma$ (see \eqref{eq:normZnk}), \eqref{eq:f2p_recons} also reads
\begin{align}
    \tilde{f}_{2j} =  \left(\sum_{n=0}^\infty \frac{(\mathcal{D}, \psi_{n,-j}^{\gamma,+})_{L^2_\gamma(\inward)}}
    {(\sigma_{n,0}^\gamma)^2} Z_{n,0}^\gamma\right)\ \d z^{2j}  +
 \left(\sum_{n=0}^\infty \frac{(\mathcal{D}, \psi_{n,n+j}^{\gamma,+})_{L^2_\gamma(\inward)}}
    {(\sigma_{n,0}^\gamma)^2} Z_{n,n}^\gamma\right)\ \d\zbar^{2j}. \label{eq:g2jSVD}
\end{align}
We show how to reconstruct the first summand, and a formula can be obtained for the second summand by complex conjugation. Start by interchanging pairings with summation, and recall that $Z_{n,0}^\gamma(z) = \hat g_{n,0}^\gamma z^n$ to write the first term as
\begin{align}
    \sum_{n=0}^\infty \frac{(\mathcal{D}, \psi_{n,-j}^{\gamma,+})_{L^2_\gamma(\inward)}}{(\sigma_{n,0}^\gamma)^2}Z_{n,0}^\gamma
= \left(\mathcal{D}, G_{2j}^\gamma (\cdot,\cdot;z) \right)_{L^2_\gamma(\inward)}, \qquad G_{2j}^\gamma (\cdot,\cdot;z) :=\sum_{n=0}^\infty \frac{\psi_{n,-j}^{\gamma,+}\hat g_{n,0}^\gamma z^n}{(\sigma_{n,0}^\gamma)^2}.
\end{align}
Substitute $\psi_{n,-j}^{\gamma,+}=\mu^{2\gamma+1}e^{(n+2j)i(\beta+\alpha+\pi)}\hat L_n^\gamma(\sin\alpha)$ to obtain
\begin{align}
G_{2j}^\gamma(\beta,\alpha;z) = \sum_{n=0}^\infty \frac{\psi_{n,-j}^{\gamma,+}\hat g_{n,0}^\gamma z^n}{(\sigma_{n,0}^\gamma)^2} = \mu^{2\gamma+1}e^{2ji(\beta+\alpha+\pi)}
\sum_{n=0}^\infty \frac{e^{ni(\beta+\alpha+\pi)}\hat L_n^\gamma(\sin\alpha)\hat g_{n,0}^\gamma z^n}{(\sigma_{n,0}^\gamma)^2}.
\end{align}
In Lemma \ref{lem:psinkpreimage}, it was computed that $\hat g_{n,0}^\gamma = \hat\ell_n^\gamma/(2i)^n$. From Appendix \ref{sec:normalization}, expand
\begin{align}
\hat L_n^\gamma(\sin\alpha)\hat g_{n,0}^\gamma
= C_n^{\gamma+1}(\sin\alpha)\frac{n!(2\gamma+1)!^2(n+\gamma)!}{i^n(n+2\gamma+1)!^2\gamma!}\left(\frac{2^{2\gamma+1}n!(2\gamma+1)!^2}{(n+\gamma+1)(n+2\gamma+1)!}\right)^{-1}2\pi.
\end{align}
Dividing by $(\sigma_{n,0}^\gamma)^2$ whose expression is given in \eqref{eq:I0range} yields
\begin{align}
    G_{2j}^\gamma(\beta,\alpha;z) &= \frac{\mu^{2\gamma+1}e^{2ji(\beta+\alpha+\pi)}}{2^{4\gamma+2}\gamma!^2} \sum_{n=0}^\infty (ze^{i(\beta+\alpha+\pi)}/i)^n C_n^{\gamma+1}(\sin\alpha)(n+\gamma+1) \\
    &\! \stackrel{\eqref{eq:gegenbauerseries}}{=} \frac{\mu^{2\gamma+1}e^{2ji(\beta+\alpha+\pi)}}{2^{4\gamma+2}\gamma!^2} \frac{(\gamma+1)(1+e^{i(2\beta+2\alpha)}z^2)}{((e^{i\beta}z+1)(e^{i(\beta+2\alpha+\pi)}z+1))^{\gamma+2}}.
\end{align}
which is the integral kernel of Equation \eqref{eq:kernel}.
By summing
\begin{align}
    \tilde{f}_{2j} &=  (\mathcal{D}, G_{2j}^\gamma)_{L^2_\gamma(\inward)}\ \d z^{2j} + (\mathcal{D}, \overline{G_{2j}^\gamma})_{L^2_\gamma(\inward)}\ \d\zbar^{2j}.
\end{align}
over $j=1,\dots,p$, the reconstruction formula of Equation \eqref{eq:reconeven} is obtained.

In the odd case, suppose $m=2p+1$ where $p\geq 0$. By a similar argument, 
\begin{align}
f_{2j+1} &= \left (\mathcal{D}, \sum_{n=0}^\infty \frac{\psi_{n,-j}^\gamma e^{i(\beta+\alpha+\pi)}}{(\sigma_{n,0}^\gamma)^2}Z_{n,0}^\gamma\right)\ \d z^{2j+1} + \left (\mathcal{D}, \sum_{n=0}^\infty \frac{\psi_{n,n+j}^\gamma e^{-i(\beta+\alpha+\pi)}}{(\sigma_{n,n}^\gamma)^2}Z_{n,n}^\gamma\right)\ \d \zbar^{2j+1}.
\end{align}
By the computation of the even case, the sums in the pairings equal $e^{i(\beta+\alpha+\pi)}G_{2j}^\gamma$ and $\overline{e^{i(\beta+\alpha+\pi)}G_{2j}^\gamma}$, respectively. Summing over $j=0,\dots,p$ leads to Equation \eqref{eq:reconodd}.
\end{proof}

\subsection{Inversion of $I_0 d^\gamma$ and $I_1 {\star}\d$}

\subsubsection{SVD-based inversions (case $\gamma>-1$)} ${}$
\medskip

\noindent{\bf Recovery of $f\in L^2_\gamma(\Dm)$ from $I_0 d^\gamma f$.} If $u\in L^2_{\gamma,+} (\partial_+ S\Dm)$ satisfies (a)-(b)-(c) from Theorem \ref{thm:rangechar} for some $m\ge 0$, then the reconstruction of the $\tilde{f}_0$ term can be done via direct inversion through the SVD of $I_0 d^\gamma$: by condition (c), $\Pi_0 u = \sum_{n=0}^\infty \sum_{k=0}^n a_{n,k} \psi_{n,k}^{\gamma,+}$ where the coefficients satisfy $\sum_{n=0}^\infty \sum_{k=0}^n \frac{|a_{n,k}|^2}{(\sigma_{n,k}^\gamma)^2}<\infty$, then one may find that $\Pi_0 u = I_0 d^\gamma \tilde{f}_0$, where 
\[ \tilde{f}_0 := \sum_{n=0}^\infty \sum_{k=0}^n \frac{a_{n,k}}{\sigma_{n,k}^\gamma} \wtZ_{n,k}^\gamma \in L^2_\gamma (\Dm). \]

\noindent{\bf Recovery of $h\in H^{1,-\gamma}_0(\Dm)$ from $I_1(\star\d h)$.}
By characterization \eqref{eq:rangeIperp}, 
\[ I_1(\star\d h) = \sum_{n=1}^\infty \sum_{k=1}^n c_{n,k} \psi_{n,k}^{\gamma,-}, \quad \text{where} \quad \sum_{n\ge 1} \sum_{k=1}^{n} \frac{|c_{n,k}|^2}{(\sigma_{n,k}^\gamma)^2} <\infty. \]
While an SVD of $I_1 {\star} \d \colon H^{1,-\gamma}_0(\Dm)\to L^2_{\gamma,-}(\partial_+ S\Dm)$ is unknown at the moment, there is the following alternate route: first notice that $I_1 ({\star}\d h) = I_1 d^\gamma (w_1\d z)$, where 
\begin{align}
    w_1 := \sum_{n=1}^\infty \sum_{k=1}^n \frac{c_{n,k}}{\sigma_{n,k}^\gamma} \wtZ_{n,k}^\gamma = \sum_{n=1}^\infty \sum_{k=1}^n \frac{(I_1 (\star\d h),\psi_{n,k}^{\gamma,-})}{\sigma_{n,k}^\gamma} \wtZ_{n,k}^\gamma \in L^2_\gamma(\Dm) \cap (\ker \dbar)^\perp,
    \label{eq:w1recons}
\end{align}
so $w_1$ can be recovered from the data. Since $w_1 \perp (\ker \dbar)$, Proposition  \ref{prop:decomp} gives $w_1 = (-2i) d^{-\gamma} \partial h$ for some $h\in H^{1,-\gamma}_0(\Dm)$, obtained by solving the elliptic equation
\begin{align}
    \dbar d^{-\gamma} \partial h = \frac{i}{2}\dbar w_1, \qquad h|_{\partial \Dm} = 0.
    \label{eq:hproblem}
\end{align}
We contend that the solution of the above problem is the $h$ we are looking for: indeed, the relation $w_1 = (-2i) d^{-\gamma}\partial h$ gives 
\begin{align*}
    d^{\gamma} w_1 \d z = (-2i) \partial h\ \d z = (-2i)  \left( \frac{\d h + i {\star}\d h}{2} \right) = \star\d h - i \d h. 
\end{align*}
Upon applying $I_1$, the last term vanishes and we indeed find that $I_1(\star\d h) = I_1 d^\gamma (w_1 \d z)$.

\subsubsection{Pestov-Uhlmann formulas (case $\gamma\ge 0$)}

The inversion of $I_0 d^\gamma$ on $L^2_\gamma (\Dm)$ is equivalent to the inversion of $I_0$ on $d^\gamma L^2_\gamma (\Dm) = L^2_{-\gamma} (\Dm)$. Now notice that for $\gamma\ge 0$, we have the continuous injections
\begin{align*}
    L^2_{-\gamma} (\Dm) \hookrightarrow L^2(\Dm), \quad \text{and} \quad H^{1,-\gamma}_0(\Dm) \hookrightarrow H^1_0(\Dm),
\end{align*}
in particular, on may use the Pestov-Uhlmann reconstruction formulas from \cite{Pestov2004}, see also \cite[Proposition 5]{Monard2017a}, to reconstruct $f\in L_{-\gamma}^2(\Dm)$ from $I_0 f$, or to reconstruct  $h\in H^{1,-\gamma}_0(\Dm)$ inversion of $I_1(\star\d h)$

\appendix
\section{Normalization}\label{sec:normalization}
The Gegenbauer polynomials are defined by the generating function \cite[Equation (1)]{gegenbauer}
\begin{align}
\frac{1}{(1 {-} 2wt {+} w^2)^{\gamma+1}} =: \sum_{n=0}^\infty w^{n}C_{n}^{\gamma+1}(t).
\label{eq:genfunc}
\end{align}
For each $n\geq 0$, define the polynomial 
\begin{align}
\hat L_n^{\gamma}(x) &\coloneq 
C_n^{\gamma+1}(x)\frac{n!(2\gamma+1)!}{(n+2\gamma+1)!}\left(\frac{2^{2\gamma+1}n!(2\gamma+1)!^2}{(n+\gamma+1)(n+2\gamma+1)!}\right)^{-1/2}\sqrt{2\pi}\\
&= \hat \ell_n^\gamma x^n + \cdots, \quad \hat\ell_n^\gamma = \frac{2^n(2\gamma+1)!(n+\gamma)!}{(n+2\gamma+1)!\gamma!}\left(\frac{2^{2\gamma+1}n!(2\gamma+1)!^2}{(n+\gamma+1)(n+2\gamma+1)!}\right)^{-1/2}\sqrt{2\pi}
\end{align}
normalized \cite[Equations (A.3,A.4)]{Wuensche2005} so that $\|\hat L_n^\gamma\|_{L^2([-1,1],(1-s^2)^{\gamma+1/2})}=\sqrt{2\pi}$. The \emph{fan-beam} polynomials are defined in Equation \eqref{def:psink}: $\psi_{n,k}^\gamma(\beta,\alpha) \coloneq  \mu^{2\gamma+1}e^{(n-2k)i(\beta+\alpha+\pi)}\hat L_n^\gamma(\sin\alpha)/2\pi$. By the choice of constants, it follows that $\|\psi_{n,k}^\gamma\|_{L_\gamma^2(\inward)}=1$.

A Zernike polynomial expands $Z_{n,k}^\gamma(z) = \hat g_{n,k}^\gamma \zbar^k z^{n-k}+\cdots$ where
\begin{align}
\hat g_{n,k}^\gamma = (-1)^k\frac{\hat\ell_n^\gamma}{(2i)^n}\binom{n}{k}.
\end{align}

Differentiating \eqref{eq:genfunc} with respect to $w$ gives the relation
\begin{align}
\frac{(\gamma+1)(1-w^2)}{(1-2wt+w^2)^{\gamma+2}} = \sum_{n=0}^\infty w^nC_{n}^{\gamma+1}(t)(n+\gamma+1), \label{eq:gegenbauerseries}
\end{align}
which is used to deduce the integral reconstruction kernel of Equation \eqref{eq:kernel}.

\section{Poincar\'e's inequality: proof of Lemma \ref{lem:poincare}}\label{sec:poincare}

In this appendix, we prove the Poincar\'e type inequality given in Lemma \ref{lem:poincare}. A similar inequality was proven in
\cite{hurri1990weighted}.

\begin{proof}[Proof of Lemma \ref{lem:poincare}]
In polar coordinates, $x=re^{i\theta}$ and $|\nabla|^2 = |\partial_r|^2
+ |\partial_\theta|^2/r^2$.
If $u$ is as given, then $u(1,\theta)\equiv 0$. The Fundamental Theorem of Calculus
implies
\begin{align}
u(r,\theta) = \int_1^r u_r(s,\theta) \d s = -\int_r^1 u_r(s,\theta)\d s.
\end{align}
After re-writing the integrand above, the Cauchy-Schwarz inequality implies
\begin{align}
|u(r,\theta)|^2 &=\left|\int_r^1 \left(u_r(s,\theta)(1-s^2)^{\gamma/2}\right)(1-s^2)^{-\gamma/2} \d s\right|^2\\
&\leq \int_r^1|u_r(s,\theta)|^2(1-s^2)^\gamma \d s\int_r^1(1-s^2)^{-\gamma} \d s.
\end{align}
Evaluate the supremum
\begin{align}
C_0 \coloneq \sup_{s\in[0,1]} (1+s)^{-\gamma} = \begin{cases}
2^{-\gamma}, & \gamma<0,\\
1, & \gamma \geq 0.
\end{cases}
\end{align}
Estimate the integral above by extracting $C_0$ and evaluate
the anti-derivative of $(1-s)^{-\gamma}$ to obtain
\begin{align}
\int_r^1(1-s^2)^{-\gamma}\d s &= \int_r^1 (1+s)^{-\gamma} (1-s)^{-\gamma} \d s\\
&\leq C_0\int_r^1(1-s)^{-\gamma}  \d s = C_0\left(-\frac{1}{-\gamma+1}(1-s)^{-\gamma+1}\right)_r^1.
\end{align}
Because the exponent $-\gamma+1$ is greater than zero, the
quantity $(1-s)^{-\gamma+1}$ vanishes when $s=1$, leading to the estimate
\begin{align}
\int_r^1(1-s^2)^{-\gamma}\d s \leq \frac{C_0}{-\gamma+1}(1-r)^{-\gamma+1}.
\end{align}
The inequality $r\leq 1$ implies $r^2 \leq r$ and $1-r\leq 1-r^2$.
The positivity of $-\gamma+1$ also implies
$(1-r)^{-\gamma+1} \leq (1-r^2)^{-\gamma+1}$. Therefore,
\begin{align}
|u(r,\theta)|^2 \leq \frac{C_0}{-\gamma+1}(1-r^2)^{-\gamma+1} \int_r^1|u_r(s,\theta)|^2(1-s^2)^\gamma \d s,    
\end{align}
which implies
\begin{align}
|u(r,\theta)|^2(1-r^2)^\gamma \leq \frac{C_0}{-\gamma+1}(1-r^2)\int_r^1|u_r(s,\theta)|^2(1-s^2)^\gamma \d s.
\end{align}
Multiply both sides of this inequality by $r$ and integrate over
$r\in[0,1]$ and $\theta\in[0,2\pi]$. After applying
Fubini's theorem in the variables $s$ and $r$, one obtains
\begin{align}
\int_0^{2\pi}\!\int_0^1 |u(r,\theta)|^2(1-r^2)^\gamma r\d r\d\theta
&\leq \frac{C_0}{-\gamma+1}\int_0^{2\pi}\int_0^1\int_r^1
	|u_r(s,\theta)|^2(1-s^2)^\gamma ds (1-r^2)r \d r\d\theta\\
&= \frac{C_0}{-\gamma+1}\int_0^{2\pi}\int_0^1\int_0^s
	|u_r(s,\theta)|^2(1-s^2)^\gamma (1-r^2)r \d r\d s\d\theta\\
&= \frac{C_0}{-\gamma+1}\int_0^{2\pi}\int_0^1
	|u_r(s,\theta)|^2(1-s^2)^\gamma
	\left(\frac{s^2}{2}-\frac{s^4}{4}\right)\d s \d\theta\\
&\leq \frac{C_0}{-\gamma+1}\int_0^{2\pi}\int_0^1
	|u_r(s,\theta)|^2(1-s^2)^\gamma s \d s\d\theta.
\end{align}
The integral on the left-hand side equals the squared norm
$\|u\|_{L^2_{\gamma}(\mathbb{D})}^2$ and the right-hand
side is a multiple of $\|u_r\|_{L^2_\gamma(\mathbb{D})}^2\leq\|\nabla u\|_{L^2_\gamma(\mathbb{D})}^2$, which implies the estimate
\begin{align}
\|u\|_{L^2_{\gamma}(\mathbb{D})} \leq C\|\nabla u\|_{L^2_\gamma(\mathbb{D})},
\end{align}
where the constant $C>0$ only depends on $\gamma$ as in the equation
$C^2 = C_0/(-\gamma+1)$.
\end{proof}

\subsection*{Acknowledgement.} The authors acknowledge partial funding from NSF-CAREER grant DMS-1943580.


\end{document}